\documentclass[final]{elsarticle} 

\usepackage{amsmath}
\usepackage{graphicx, indentfirst, amsmath, amsfonts, amssymb, amsthm, tabularx, multirow,rotating}

\def\l1{{\lambda}_1}

\newcommand{\KS}{Kuramoto-Sivashinsky Equation}

\newcommand{\KSII}{two-dimensional Kuramoto-Sivashinsky Equation}

\newcommand{\pd}{\partial}

\def\x1{{\xi }_{xx}}
\def\x2{{\xi }_{yy}}
\def\x3{{\xi }_{xy}}

\def\e1{{\eta }_{xx}}
\def\e2{{\eta }_{yy}}
\def\e3{{\eta }_{xy}}
\def\h{\frac{\partial }{\partial x}}

\newtheorem{thm}{Theorem}[section]

\newdefinition{defn}{Definition}[section]

\theoremstyle{remark}
\newtheorem{case}{Case}
\newtheorem{subcase}{$\bullet$ Subcase}[case]
\newtheorem{rem}[defn]{Remark}

\newtheoremstyle{syms}
  {3pt}		
  {3pt}		
  {}			
  {}			
  {\bfseries}	
  {:}			
  {\newline}	
  {}			

\theoremstyle{syms}
\newtheorem*{pt}{$\boldsymbol{\partial_t}$}
\newtheorem*{px}{$\boldsymbol{\partial_x}$}
\newtheorem*{py}{$\boldsymbol{\partial_y}$}
\newtheorem*{pxpy}{$\boldsymbol{y\partial_x-x\partial_y}$}

\newcommand{\beqn}{\begin{eqnarray*}}
\newcommand{\eeqn}{\end{eqnarray*}}
\newcommand{\beqnn}{\begin{eqnarray}}
\newcommand{\eeqnn}{\end{eqnarray}}

\newcommand{\bb}{\begin{equation}}
\newcommand{\ee}{\end{equation}}
\newcommand{\ba}{\begin{array}}
\newcommand{\ea}{\end{array}}

\makeatletter
\def\ps@pprintTitle{%
  \let\@oddhead\@empty
  \let\@evenhead\@empty
  \def\@oddfoot{\reset@font\hfil\thepage\hfil}
  \let\@evenfoot\@oddfoot
}
\makeatother

\begin{document}


\title{ Group Classification and Conservation Laws for  a two-dimensional Generalized Kuramoto-Sivashinsky Equation}

\author[imecc]{Y.~Bozhkov}
\ead{bozhkov@ime.unicamp.br}
\author[imecc]{S.~Dimas\corref{cor}}
\ead{spawn@math.upatras.gr }
\address[imecc]{Instituto de Matem\'atica, Estat\'\i stica e Computa\c c\~ao Cient\'\i fica - IMECC\\
Universidade Estadual de Campinas - UNICAMP\\
Rua S\'ergio Buarque de Holanda, 651\\
$13083$-$970$ - Campinas - SP, Brazil}

\cortext[cor]{Corresponding author}


\begin{abstract}
	The two-dimensional  an\-iso\-tro\-pic Kuramoto-Si\-va\-shin\-sky equation is a  forth-order nonlinear evolution equation in two spatial dimensions that arises in sputter erosion and epitaxial growth on vicinal surfaces. A generalization of this equation is proposed and studied via group analysis methods. The complete group classification of this generalized Kuramoto-Sivashinsky equation is carried out, it is classified according to the property of the self-adjointness and the corresponding conservation laws are established. 
\end{abstract}

\maketitle


\section{Introduction}\vphantom{\cite{EoM22687}}

The celebrated \KS\ (KSE) 
\begin{equation}\label{KS1}
	u_t+u_{xxxx}+u_{xx}+\frac{1}{2}u_x^2=0,
\end{equation}
where $u=u(x,t)$, is an equation that for nearly half a century has attracted the attention of many researchers from various areas due to its simple and yet rich dynamics. It first appeared in mid-1970s by Kuramoto in the study of angular-phase turbulence for a system of reaction--diffusion equations modeling the Belou\-sov--Zhabotinskii reaction in three spatial  dimensions \cite{Ku78,KuTsu75, KuTsu76} and independently by Sivashinsky in the study of hydrodynamic instabilities in laminar flame fronts \cite{Si77, MiSi77,Si80}. 

In a physical context equation \eqref{KS1} is used to model continuous media that exhibits chaotic behavior such as weak turbulence on interfaces between complex flows (quasi-planar flame front and the fluctuation of the positions of a flame front, fluctuations in thin viscous fluid films flowing over inclined planes or vertical walls, dendritic phase change fronts in binary alloy mixtures),  small perturbations of a metastable planar front or interface (spatially uniform oscillating chemical reaction in a homogeneous medium) and physical systems driven far from the equilibrium due to intrinsic instabilities (instabilities of dissipative trapped ion modes in plasmas and phase dynamics in reaction-diffusion systems) \cite{CoKroTaRo76, LaMaRuTa75, Ma88,SiMi80, HaZha95, CroHo93}.

As a dynamical system the KSE is known for its chaotic solutions and complicated behavior due to the terms that appear. Namely, the $u_{xx}$ term acts as an energy source and has a destabilizing effect at large scale, the dissipative $u_{xxxx}$ term provides dumping in small scales and, finally, the nonlinear term provides  stabilization by transferring energy between large and small scales. Because of this fact, equation \eqref{KS1} was studied extensively as a paradigm of finite dynamics in a partial differential equation (PDE). Its multi-modal, oscillatory and chaotic  solutions have been investigated \cite{HyNi86,HyNiZa86,KeNiSco90, LvoLePaPro93, Mi86}, its non-integrability was established via its Painlev\'e analysis \cite{CoMu89, NiScheTe85,ThuFri86} and due to its bifurcation behavior a connection to low finite-dimensional dynamical systems is established \cite{Te97, WiHo99}. 

The generalization of KSE to two dimensions comes naturally, the \KSII,
 \begin{equation}\label{KS2}
	u_t+\nabla^4u+\nabla^2u+(\nabla u) \cdot (\nabla u)=0,
\end{equation}
where now $u=u(x,y,t)$ and $\nabla^2=\nabla\cdot\nabla,\ \nabla^4=\nabla\cdot\nabla(\nabla\cdot\nabla)$. Eq.~\eqref{KS2} has equally attracted much attention because of the same spatiotemporal chaos properties that exhibits and its applications in modeling complex dynamics in hydrodynamics \cite{BoChoHwa99,DroZhaLuWa99,  GaMiPo2k9,LuBo98}.  Nevertheless, due to the additional spatial dimension equation \eqref{KS2} is very challenging and even its well-posedness is still an open problem \cite{De2k9,JaHaPa93}. 

One generalization of equation \eqref{KS2} of much interest is the \textit{anisotropic} \KSII,
 \begin{equation}\label{aKS2}
	u_t=\frac{1}{2}u_x^2+\frac{\beta}{2}u_y^2-u_{xx}-\alpha u_{yy}-u_{xxxx}-2u_{xxyy}-u_{yyyy},
\end{equation}
where the two real parameters $\alpha,\beta$ control the anisotropy of the linear and the nonlinear term, respectively. In other words, the stability of the solutions of equation \eqref{aKS2}. The anisotropic \KSII, due to the fact that it describes linearly unstable surface dynamics in the presence of in-plane anisotropy, has a wide range of applications. For instance, as a model for the nonlinear evolution of sputter-eroded surfaces and describing the epitaxial growth of a vicinal surface destabilized by step edge barriers, for further details see \cite{RoKru95} and the refereces therein, in particular \cite{MaBa97}.

This paper focuses in the following generalization of the anisotropic KSE \eqref{aKS2}:
\begin{equation}\label{gKS2}
	u_{  t}=\frac{1}{2}u_x^2+h(u) u_y^2+r(u)u_{xx}+g(u) u_{  yy}-u_{  xxxx}-2 u_{  xxyy}-u_{  yyyy}+f(u)
\end{equation}
and its study under the prism of Lie point symmetries and conservation laws. (Here, $f,h,g$ and $r$ are considered smooth functions of $u=u(x,y,t)$.)

The symmetries of a differential equation are of fundamental importance since they are a structural property of the equation. In addition, finding the symmetries of a differential equation is an analytic method that can be applied to integrable and non-integrable equations alike. Nevertheless the symmetry analysis is constrained to rudimentary generalizations of equation \eqref{KS2} \cite{NaAha2k11, NaAha2k12}. Things are worse when one looks for conservation laws. For the complexity of the calculations involved, the research is constrained to generalizations of the one dimensional \KS\ \cite{BruGaIbra2k9}.

In this frame two different classifications are performed: the complete group classification and a classification with respect to the property of self-ad\-joint\-ness. Having the symmetries for each possible case and the self-adjoint cases at hand, the conservations laws for that system by using the Noether operator $\mathcal{N}$ are obtained, see also  \cite{Ibra2k7, Ibr2k11a, Ibr2k11b}.

Calculating the symmetries of the system \eqref{aKS2}, obtaining its adjoint system and applying the Noether operator to obtain the conserved vectors are well-defined algorithmic procedures. Nevertheless,  the calculations involved are usually very difficult and extensive even for the simplest equations.  Thus, it may become very tedious and error prone. For that reason the use  of computer algebra systems like {\it Mathematica, Maple, Reduce}, etc. and of special symbolic packages that are build based on them is very crucial. For this work the \textit{Mathematica} package SYM \cite{DiTs2k5a,DiTs2k6,Dimas2k8} was extensively used for all the results that follow. Namely, for obtaining the symmetries of the system, to get and simplify the adjoint system and the conserved vectors that emerge from the use of the Noether operator.

In Section 2 the definitions and the analytical tools used are introduced. Section 3 explores the self-adjointness of equation \eqref{gKS2}. Then, the complete group classification of eq. \eqref{gKS2} is carried out in Section 4 followed by Section 5 where the conservation laws are established. Finally in Section 6 some comments and concluding remarks are presented.

\section{Notation and methodology}

We shall employ in this work the two analytical tools: the symmetry analysis and the use of the Noether operator identity for the explicit construction of conservation laws. For both the necessary definitions  of the notions that will be encountered  in the main body of this work are illustrated below, adapted accordingly to the needs of the present paper. 

For brevity, we denote:
\begin{equation*}
	\begin{split}
		&\Delta(x,y,t,u,u_x,\dots,u_{yyyy}) = \\
		&\quad\frac{1}{2}u_x^2+h(u) u_y^2+r(u)u_{xx}+g(u) u_{  yy}-u_{  xxxx}-2 u_{  xxyy}-u_{  yyyy}-u_{  t}+f(u).
	\end{split}
\end{equation*}

\subsection{Modern group analysis}

The symmetry or modern group analysis is a valuable analytic tool for the investigation of differential equations. For a full treatise of the subject there is a wealth of classical texts that encompass all aspects of the theory \cite{BluAncChe2k9,BluKu89,Hydon2k, Ibra2k1, Olver2k, Ovsiannikov82, Stephani90} . 
\begin{defn}
	Let the differential operator,
	 \bb\label{sg}
	 	X={\xi }^1(x,y,t,u)\pd_x+{\xi }^2(x,y,t,u)\pd_y+{\xi }^3(x,y,t,u)\pd_t+ \eta(x,y,t,u)\pd_u.
	\ee
	 This operator, from now on called \textit{infinitesimal generator}, determines a \textit{Lie point symmetry} of equation \eqref{gKS2}, if and only if, its action on the equation will be, modulo the equation itself,  identically zero, that is:
\bb\label{lsc}
	\left.X^{(4)}\left[\Delta(x,y,t,u,u_x,\dots,u_{yyyy})\right]\right\rvert_{\Delta(x,y,t,u,u_x,\dots,u_{yyyy})=0}\equiv 0,
\ee
where $X^{(4)}$ is the fourth order prolongation of the operator $X$ given by 
\bb\label{prolong}
	X^{(4)}= X +\sum_{s=1}^4{\eta }^{(s)}_{i_1\dots i_s}\frac{\partial}{\partial u_{i_1\dots i_s}},\ i_n=1,2,3
\ee
with
 \begin{align*}
 	{\eta }^{(1)}_{i}&=D_i\eta - (D_i{\xi }^j)u_j,& {\eta }^{(s)}_{i_1\dots i_s} &=D_{i_s} {\eta }^{(s-1)}_{i_1\dots i_{s-1}} - (D_{i_s}{\xi }^{j})u_{i_1\dots i_{s-1} j}
 \end{align*}
 and the partial derivatives denoted by
 $$
 	u_i=\frac{\partial u}{\partial x^i},\, (x^1,x^2,x^3)=(x,y,t).
 $$
\end{defn}
From the condition \eqref{lsc}, called linearized symmetry condition, an overdetermined system of linear partial differential equations emerges. By solving this system, called the determining equations, we determine the coefficients ${\xi }^i,\, \eta$ of the infinitesimal generator. Hence, the point symmetries of the equation. The group classification occurs in that phase: the determining equations contain also  the functions $f, g,h,r$. The group classification is performed by investigating each case where specific relations among the unknown elements remove equations from the set of determining equations, and by doing that enlarging the set of solutions. 

\subsection{The adjoint and self-adjoint concept and conservation laws}

In accordance to \cite{Ibra2k7,Ibr2k11a,Ibr2k11b} we introduce the required notions that will enable us  to construct conservation laws for the equation \eqref{gKS2}.
\begin{defn}
	The \textit{adjoint equation} to equation \eqref{gKS2} is
	\bb\label{adjEq}
		\frac{\delta\mathcal L}{\delta u}=0,
	\ee
	where $\mathcal L$ is the \textit{formal Lagrangian} given by
	$$
		\mathcal L=\upsilon(x,y,t)\Delta(x,y,t,u,u_x,\dots,u_{yyyy}).
	$$
	Here $\upsilon$ is a new dependent variable, called also \emph{nonlocal variable},  and $\delta/\delta u$ is the Euler-Lagrange operator
	$$
		\frac{\delta}{\delta u}=\frac{\partial}{\partial u}+\sum_{i_x+i_y+i_t=1}^\infty(-1)^{i_x+i_y+i_t} D_x^{i_x}D_y^{i_y}D_t^{i_t}\frac{\partial}{\partial u_{i_xx\,i_yy\,i_tt}},
	$$
	with $u_{i_xx\,i_yy\,i_tt}$ denoting $\partial^{i_x+i_y+i_t}u/\partial x^{i_x}\partial y^{i_y}\partial t^{i_t},\, i_x,i_y,i_t\ge0$.
\end{defn}

\begin{defn}
	We say that the equation \eqref{gKS2} is \textit{strictly self-adjoint} if the adjoint equation \eqref{adjEq} becomes equivalent to the equation \eqref{gKS2} after the substitution $\upsilon=u$:
	$$
		\frac{\delta\mathcal L}{\delta u} = \lambda \Delta(x,y,t,u,u_x,\dots,u_{yyyy}),
	$$
	with $\lambda$  a generic coefficient.	
\end{defn}

\begin{defn}
	We say that the equation \eqref{gKS2} is \textit{quasi self-adjoint} if the adjoint equation \eqref{adjEq} becomes equivalent to the equation \eqref{gKS2} after the substitution $\upsilon=\phi(u)$, where $\phi(u)\ne0$.	
\end{defn}

\begin{defn}
	We say that the equation \eqref{gKS2} is \textit{nonlinearly self-adjoint} if the adjoint equation \eqref{adjEq} becomes equivalent to the equation \eqref{gKS2} after the substitution $\upsilon=\phi(x,y,t,u)$, where $\phi(x,y,t,u)\ne0$.
\end{defn}

\begin{rem}
	The concept of the nonlinear self-adjointness can be further extended by considering differential substitutions  of the form
	$$
		\upsilon=\phi(x,y,t,u,u_{(1)},\dots,u_{(r)}),
	$$
	where $u_{(r)}$ are the derivatives of $u$ of order $r$.
\end{rem}

\begin{rem}
	From the above definitions it is obvious that if an equation is strictly or quasi self-adjoint then it is also and nonlinearly self-adjoint.
\end{rem}

\begin{thm}[Explicit formula for conserved vectors]
	Let equation \eqref{gKS2} be nonlinearly self-adjoint and a \eqref{sg} its Lie point symmetry. A conserved vector can be constructed by the the following formula:
	\begin{equation}\label{cvf}
		C^i=\xi^i\mathcal L + \sum_{i_x+i_y+i_t=0}^\infty D_x^{i_x}D_y^{i_y}D_t^{i_t}(W) \frac{\delta^*\mathcal L}{\delta^* u_{i\ i_xx\ i_yy\ i_tt}},\, i=x,y,t,\ i_i\ge0
	\end{equation}
	where
	$$
		W=\eta-\xi^1u_x-\xi^2u_y-\xi^3u_t,
	$$
	$\delta^*/\delta^*u$ is the ``weighted" Euler-Lagrange operator
	\begin{equation*}
		\begin{split}
			&\frac{\delta^*\mathcal L}{\delta^* u_{i_xx\, i_yy\, i_tt}}=\frac{\partial}{\partial u}\\
			&+\sum_{s=j_x+j_y+j_t=1}^\infty(-1)^{s} \frac{\binom{s}{j_x,j_y,j_t}}{\binom{s+j_x+j_y+j_t}{i_x+j_x,i_y+j_y,i_t+j_t}}D_x^{j_x}D_y^{j_y}D_t^{j_t}\frac{\partial}{\partial u_{(i_x+j_x)x\,(i_y+j_y)y\,(i_t+j_t)t}}
		\end{split}
	\end{equation*}
	with $\binom{N}{i_1,i_2,\dots,i_r}=\frac{N}{i_1!i_2!\dots i_r!},\, N=i_1+i_2+\cdots+i_r$, the multinomial and $\mathcal L$ the formal Lagrangian after substituting with $\upsilon=\phi(x,y,t,u)$.
\end{thm}
\begin{proof}
	For a proof see \cite{Ibra2k7}.
\end{proof}

\begin{rem}
	From a conserved vector $(C^x,C^y,C^t)$ the conservation law has the form
	$$
		D_x(C^x)+D_y(C^y)+D_t(C^t)=0
	$$
	satisfied on the solutions of equation \eqref{gKS2}.
\end{rem}

\begin{defn}
	A conserved vector is called \textit{trivial} if,
	\begin{itemize}
		\item its divergence is identically zero and,
		\item the components of the vector vanish on the solutions of equation  \eqref{gKS2}.
	\end{itemize}
\end{defn}

\begin{rem}
	Only the nontrivial conserved vectors will be considered.
\end{rem}

\section{The self-adjointness classification}\label{sec:sa}

As it can be seen by the definitions of self-adjointness  the first step is to obtain the adjoint equation. For equation  \eqref{gKS2} the adjoint equation is
\begin{equation}\label{adjoint}
	\begin{split}
		&2 (h(u)-g') u_{y} \upsilon _{y}-g(u) \upsilon _{yy}- r(u) \upsilon _{xx}+(1-2 r') u_{x} \upsilon _{x}\\
		&\quad-\upsilon \left(f'+(g^{\prime \prime }-h' u_{y}) u_{y}{}^2+2 (g' - h(u))u_{yy}+r^{\prime \prime } u_{x}{}^2+(2 r' -1)u_{xx}\right)\\
		&\qquad +\upsilon _{yyyy}+2 \upsilon _{xxyy}+\upsilon _{xxxx}-\upsilon _{t}=0
	\end{split}
\end{equation}
where $\upsilon=\upsilon(x,y,t)$ is the new dependent variable.

\begin{thm}
	Equation \eqref{gKS2} has no strictly self-adjoint subcase.
\end{thm}
\begin{proof}
	 By making the substitution $\upsilon=u$ in the adjoint equation \eqref{adjoint}  and then  substituting $u_t$ from equation \eqref{gKS2},  we get
	 \begin{equation*}
	 	\begin{split}
	 		&\left(h(u)-2 g'\right) u_{y}{}^2-2 g(u) u_{yy}+2 u_{yyyy}+\frac{1}{2}\left(1-4 r'\right) u_{x}{}^2-2 r(u) u_{xx}\\
			&\quad -u \left(f'+\left(g^{\prime \prime -h'}\right) u_{y}{}^2+2 \left(g'-h(u)\right) u_{yy}+r^{\prime \prime } u_{x}{}^2+\left(-1+2 r'\right) u_{xx}\right)\\
			&\qquad +4 u_{xxyy}+2 u_{xxxx}-f(u)\equiv0.
	 	\end{split}
	 \end{equation*}
	 It is obvious that there is no choice of the functions $f,h,g,r$ satisfying  the above condition. In other words there is no choice of the functions $f,h,g,r$ that will make equation  \eqref{gKS2} strictly self-adjoint.
\end{proof}

\begin{thm}\label{thm:1}
	Equation \eqref{gKS2} is quasi self-adjoint when $f=c_1, r=u/2+c_2, h=g^\prime$.
\end{thm}
\begin{proof}
 	After making the substitution $\upsilon=\phi(u)$ in the adjoint equation \eqref{adjoint}  and then  substituting  $u_t$ from equation \eqref{gKS2},  we get
	 \begin{equation*}
	 	\begin{split}
			& \left(\left(h(u)-2 g'\right) \phi '+\phi (u) \left(h'-g^{\prime \prime }\right)-g(u) \phi ^{\prime \prime }+6 \phi ^{\prime \prime \prime } u_{yy}+2 \phi ^{\prime \prime \prime } u_{xx}\right)u_{y}{}^2\\
			&\quad+\left(\phi (u) \left(1-2 r'\right)-2 r(u) \phi '+2 \phi ^{\prime \prime } u_{yy}\right) u_{xx}+4 \phi ' u_{xxyy}+2 \phi ^{\prime \prime \prime \prime } u_{y}{}^2u_{x}{}^2\\
			& \quad\,+\left(\frac{1}{2} \left(\left(1-4 r'\right) \phi '-2 \left(\phi (u) r^{\prime \prime }+r(u) \phi ^{\prime \prime }\right)\right)+2 \phi ^{\prime \prime \prime }( u_{yy}+3 u_{xx})\right)u_{x}{}^2\\			
			&\quad\,\,+3 \phi ^{\prime \prime } u_{yy}{}^2+2 \phi ' u_{yyyy}+\phi ^{\prime \prime \prime \prime } u_{x}{}^4+4 \phi ^{\prime \prime } u_{xy}{}^2+3 \phi ^{\prime \prime } u_{xx}{}^2+2 \phi ' u_{xxxx}\\
			&\quad\,\,\,-(\phi (u) f'+f(u) \phi ')+\phi ^{\prime \prime \prime \prime } u_{y}{}^4+\left(2 \phi (u) \left(h(u)-g'\right)-2 g(u) \phi '\right) u_{yy}\\
			&\quad\,\,\,\,+4 \phi ^{\prime \prime }\left( u_{yyy}+ u_{xxy}\right)u_{y} + 4\left(2 \phi ^{\prime \prime \prime } u_{y} u_{xy}+ \phi ^{\prime \prime } u_{xyy}+ \phi ^{\prime \prime } u_{xxx}\right)u_{x}\equiv0.
	 	\end{split}
	 \end{equation*}
	 This condition must vanish for every solution $u$ of equation \eqref{gKS2}. Hence we arrive at the system:
	$$
 		\phi (u) f'+f(u) \phi '=0,
	$$
	$$
		\phi (u) \left(1-2 r'\right)-2 r(u) \phi '=0,
	$$
	$$
		\h(u) \phi (u)=\phi (u) g'+g(u) \phi ',
	$$
	 $$
	 	\phi '=0,\,\phi ^{\prime \prime }=0,\,\phi ^{\prime \prime \prime }=0,\,\phi ^{\prime \prime \prime \prime }=0,
	$$	
	$$	
 		\left(1-4 r'\right) \phi '-2 \left(\phi (u) r^{\prime \prime }+r(u) \phi ^{\prime \prime }\right)=0,
	$$
	$$
 		h(u) \phi '-2 g' \phi '+\phi (u) \left(h'-g^{\prime \prime }\right)-g(u) \phi ^{\prime \prime }=0.
	$$	
	Solving the above system we get that $\phi(u)=c\ne0$ and as stated, $f=c_1, r=u/2+c_2, h=g^\prime$.
\end{proof}

\begin{thm}\label{thm:2}
	Equation \eqref{gKS2} is nonlinear self-adjoint if and only if:
	\begin{enumerate}
		\item $r=u/2+\alpha, g=\beta u+\gamma, h=\beta, f=\delta u^2+\epsilon u +\zeta,$
		\item $r=u/2+\alpha,  h=g^\prime, f=\beta u^2+\gamma u +\delta + c \int g(u)\,du,\ \beta,c,g^{\prime\prime}\ne0,$
		\item $r=u/2+\alpha,  h=g^\prime, f=\beta  u +\gamma + c \int g(u)\,du,\ c,g^{\prime\prime}\ne0,$
		\item $r=u/2+\alpha,  h=g^\prime, f=\beta u^2+\gamma u +\delta ,\ \beta,g^{\prime\prime}\ne0,$
		\item $r=u/2+\alpha,  h=g^\prime, f=\beta  u +\gamma,\ g^{\prime\prime}\ne0,$
	\end{enumerate}
\end{thm}
\begin{proof}
	After making the substitution $\upsilon=\phi(x,y,t,u)$ in the adjoint equation \eqref{adjoint}  and then  substituting $u_t$ from equation \eqref{gKS2},  we get
	\begin{equation*}
	 	\begin{split}
			&-\phi (x,y,t,u) f'-f(u) \phi _{u}+2 u_{yyyy} \phi _{u}+4 u_{xxyy} \phi _{u}+2 u_{xxxx} \phi _{u}+3 u_{yy}{}^2 \phi _{uu}\\
			&+4 u_{xy}{}^2 \phi _{uu}+3 u_{xx}{}^2 \phi _{uu}+u_{y}{}^4 \phi _{uuuu}+u_{x}{}^4 \phi _{uuuu}-\phi _{t}+4 u_{yyy} \phi _{yu}+4 u_{xxy} \phi _{yu}\\
			&+4 u_{y}{}^3 \phi _{yuuu}-g(u) \phi _{yy}+\phi _{yyyy}+4 u_{xyy} \phi _{xu}+4 u_{xxx} \phi _{xu}+4 u_{x}{}^3 \phi _{xuuu}\\
			&+u_{xx} \left(\phi (x,y,t,u) \left(1-2 r'\right)-2 r(u) \phi _{u}+2 \phi _{yyu}+6 \phi _{xxu}\right)+2 \phi _{xxyy}+\phi _{xxxx}\\
			&+u_{yy} \left(2 u_{xx} \phi _{uu}+2 u_{x}{}^2 \phi _{uuu}+4 u_{x} \phi _{xuu}+2 \left(\phi (x,y,t,u) \left(h(u)-g'\right)\right.\right.\\
			&\quad-\left.\left.g(u) \phi _{u}+3 \phi _{yyu}+\phi _{xxu}\right)\right)+8 u_{xy} \phi _{xyu}-r(u) \phi _{xx}\\
			&+u_{y}{}^2 \left(\phi (x,y,t,u) \left(h'-g^{\prime \prime }\right)+\left(h(u)-2 g'\right) \phi _{u}-g(u) \phi _{uu}+6 u_{yy} \phi _{uuu}\right.\\
			&\quad+\left.2 u_{xx} \phi _{uuu}+2 u_{x}{}^2 \phi _{uuuu}+6 \phi _{yyuu}+4 u_{x} \phi _{xuuu}+2 \phi _{xxuu}\right)\\
			&+u_{x}{}^2 \left(-\phi (x,y,t,u) r^{\prime \prime }+\left(\frac{1}{2}-2 r'\right) \phi _{u}-r(u) \phi _{uu}+6 u_{xx} \phi _{uuu}+2 \phi _{yyuu}\right.\\
			&\quad+\left.6 \phi _{xxuu}\right)\\
			&+u_{y} \left(4 u_{yyy} \phi _{uu}+4 u_{xxy} \phi _{uu}+2 \left(h(u)-g'\right) \phi _{y}-2 g(u) \phi _{yu}+12 u_{yy} \phi _{yuu}\right.\\
			&\quad+4 u_{xx} \phi _{yuu}+4 u_{x}{}^2 \phi _{yuuu}+8 u_{xy} \phi _{xuu}+u_{x} \left(8 u_{xy} \phi _{uuu}+8 \phi _{xyuu}\right)\\
			&\qquad+\left.4 \left(\phi _{yyyu}+\phi _{xxyu}\right)\right)\\
			&+u_{x} \left(4 u_{xyy} \phi _{uu}+4 u_{xxx} \phi _{uu}+8 u_{xy} \phi _{yuu}+\left(1-2 r'\right) \phi _{x}-2 r(u) \phi _{xu}\right.\\
			&\quad+\left.12 u_{xx} \phi _{xuu}+4 \left(\phi _{xyyu}+\phi _{xxxu}\right)\right)\equiv0.
	 	\end{split}
	 \end{equation*}
	  This condition must vanish for every solution $u$ of equation \eqref{gKS2}. Hence we attain the system:
	$$
		\left(1-2 r'\right) \phi _{x}-2 r(u) \phi _{xu}+4 \left(\phi _{xyyu}+\phi _{xxxu}\right)=0 ,
	$$
	$$
		2 (h(u)- g' )\phi _{y}-2 g(u) \phi _{yu}+4 \phi _{yyyu}+4 \phi _{xxyu}=0 ,
	$$
	$$
		(h(u)- g' ) \phi (x,y,t,u)+3 \phi _{yyu}+\phi _{xxu}-g(u) \phi _{u}=0 ,
	$$
	$$
		 \left(1-2 r'\right)\phi (x,y,t,u)-2 r(u) \phi _{u}+2 \phi _{yyu}+6 \phi _{xxu}=0 ,
	$$
	$$
		\phi _{yuu}=0 ,\,	\phi _{uuu}=0,\,	\phi _{xuuu}=0 ,\,\phi _{yuuu}=0 ,\,\phi _{uuuu}=0 ,
	$$
	$$
	   	\phi _{u}=0 ,\,\phi _{xu}=0 ,\,\phi _{yu}=0 ,\,\phi _{uu}=0 ,\,\phi _{xuu}=0 ,\,\phi _{xyu}=0 ,
	$$	
	$$		
		\left(1-4 r'\right) \phi _{u}-2 r(u) \phi _{uu}-2 r^{\prime \prime } \phi (x,y,t,u)+4 \phi _{yyuu}+12 \phi _{xxuu}=0 ,
	$$
	$$		
		f'\phi +f(u) \phi _{u}+\phi _{t}+g(u) \phi _{yy}-\phi _{yyyy}+r(u) \phi _{xx}-2 \phi _{xxyy}-\phi _{xxxx}=0,
	$$
	$$
 		 \left(h'-g^{\prime \prime }\right)\phi (x,y,t,u)+h(u) \phi _{u}-2 g' \phi _{u}-g(u) \phi _{uu}+6 \phi _{yyuu}+2 \phi _{xxuu}=0.
	$$
	Looking at the above system it is obvious that 
	\begin{equation}\label{firstsol}
		\phi=\phi(x,y,t),\ r(u)=u/2+c,\ h=g^\prime.
	\end{equation}
	By substituting \eqref{firstsol} in the system, we arrive to the following equation:
	\begin{equation}\label{condition}
		\phi f'+\phi _{, t}+g(u) \phi _{, yy}-\phi _{, yyyy}+\frac{1}{2} \left(u+2 \boldsymbol{c}\right) \phi _{, xx}-2 \phi _{, xxyy}-\phi _{, xxxx}=0.
	\end{equation}
	By differentiating equation \eqref{condition} two times with respect to $u$ we have in addition
	\begin{align*}
		&\phi f^{\prime \prime \prime }+g^{\prime \prime } \phi _{, yy}=0,\\
		&\phi f^{\prime \prime }+g' \phi _{, yy}+\frac{\phi _{, xx}}{2}=0.
	\end{align*}
	By using these two equations in combination with equation \eqref{condition} we arrive at the five cases stated.
\end{proof}
Therefore having  obtained the specific self-adjoint subcases of equation \eqref{gKS2} it  remains  to know their symmetries in order to  construct the corresponding conserved vectors. In the next Section the complete group classification for equation \eqref{gKS2}  is given.

\section{The Group Classification}

From the invariant surface condition \eqref{lsc} for equation  \eqref{gKS2} the system of the determining equations is obtained, see Appendix~\ref{app:det}. By solving the subsystem not containing any of the functions $f,g,h,r$ we obtain:  
\begin{align*}
	\xi^1(x,y,t,u) &= \mathcal{F}_{11}(t)-y \mathcal{F}_{13}(t)+\frac{x }{4}\mathcal{F}_1'(t),\\
	\xi^2(x,y,t,u)  &=\mathcal{F}_{12}(t)+x \mathcal{F}_{13}(t)+\frac{y }{4}\mathcal{F}_1'(t),\\
	\xi^3(x,y,t,u) &=\mathcal{F}_1(t),\\
	\eta(x,y,t,u) &= \mathcal{F}_{15}(y,t)-\frac{u }{2}\mathcal{F}_1'(t)-\frac{1}{8} x \left(8 \mathcal{F}_{11}'(t)-8 y \mathcal{F}_{13}'(t)+x \mathcal{F}_1{}^{\prime \prime }(t)\right),
\end{align*}
and the remaining determining equations are:
$$
 	r' \mathcal{F}_1{}^{\prime \prime } =0,
$$
$$
 	h' \mathcal{F}_1{}^{\prime \prime } =0,
$$
$$
 	g' \mathcal{F}_1{}^{\prime \prime } =0,
$$
$$
 	r' \mathcal{F}_{13}' =0,
$$
$$
 	h' \mathcal{F}_{13}' =0,
$$
$$
 	g' \mathcal{F}_{13}' =0,
$$
$$
 	r' \mathcal{F}_{11}' =0,
$$
$$
 	h' \mathcal{F}_{11}' =0,
$$
$$
 	g' \mathcal{F}_{11}' =0,
$$
$$
 	f' \mathcal{F}_1{}^{\prime \prime }-\mathcal{F}_1{}^{\prime \prime \prime } =0,
$$
$$
 	f' \mathcal{F}_{11}'-\mathcal{F}_{11}{}^{\prime \prime } =0,
$$
$$
 	f' \mathcal{F}_{13}'-\mathcal{F}_{13}{}^{\prime \prime } =0,
$$
$$
 	(1+2 h(u)) \mathcal{F}_{13}' =0,
$$
$$
 	(2 h(u)-1) \mathcal{F}_{13}=0,
$$
$$
 	h' \left(2 \mathcal{F}_{15}-u \mathcal{F}_1'\right) =0,
$$
$$
 	(g(u)-r(u)) \mathcal{F}_{13} =0,
$$
$$
 	4 \mathcal{F}_{12}'+y \mathcal{F}_1{}^{\prime \prime }+8 h(u) {\mathcal{F}_{15}}_{y} =0,
$$
$$
 	8 \mathcal{F}_{15} r'+4 \left(r(u)-u r'\right) \mathcal{F}_1' =0,
$$
$$
 	8 \mathcal{F}_{15} g'+4 \left(g(u)-u g'\right) \mathcal{F}_1' =0,
$$
\begin{equation*}
	\begin{split}
 		&4 \mathcal{F}_{15} f'+6 f(u) \mathcal{F}_1'-2 u f' \mathcal{F}_1'+2 u \mathcal{F}_1{}^{\prime \prime }-r(u) \mathcal{F}_1{}^{\prime \prime }-4 {\mathcal{F}_{15}}_{t}\\
		&\quad+4 g(u) {\mathcal{F}_{15}}_{yy}-4 {\mathcal{F}_{15}}_{yyyy} =0.
	\end{split}
\end{equation*}

The resulted group classification is summarized in Table \ref{tab:sym}. The first row gives the symmetries that occur for every possible choice of the functions $f,g,h$ and $r$. In each subsequent row a special case appears along with the additional symmetries it has.

{\renewcommand*{\thefootnote}{\fnsymbol{footnote}}
\renewcommand*{\thempfootnote}{\fnsymbol{mpfootnote}}
\begin{sidewaystable}[htdp]
\begin{center}
\begin{tabular}{|c|c|c|c|c|} \hline
 $f$								&  $g$ 					& $h$ 						& $r$									& \textit{Symmetries}	\\\hline\hline
 $\forall$							& $\forall$					& $\forall$						& $\forall$									& $\mathfrak X_1=\partial _t,\ \mathfrak X_2=\partial _x,\ \mathfrak X_3=\partial _y$   \\ \hline
 $\forall$							& $r$					& $\frac{1}{2}$					&  $\forall$								& $\mathfrak X_4=x\partial _y-y\partial _x$ \\\hline
								&   						&							& 										& $\mathfrak X_4=e^{\alpha t }\partial _u,\ \mathfrak X_5=x\partial _y-y\partial _x,$ \\
 $\alpha u +\beta, \alpha\ne0$			& $r$ 					&  $\frac{1}{2}$					& $c$									& $\mathfrak X_6=\frac{e^{\alpha t }}{\alpha }\partial _y-e^{\alpha t } y\partial _u,$  \\
 								&  						& 							&										& $\mathfrak X_7=\frac{e^{\alpha t }}{\alpha }\partial _x-e^{\alpha t } x\partial _u $ \\\hline					 
								&						&							&										& $\mathfrak X_4=\partial _u,\ \mathfrak X_5 = t\partial _x-x\partial _u,$\\
$\beta$							& $r$	   				& $\frac{1}{2}$					& $c\ne0$									& $\mathfrak X_6=t\partial _y-y\partial _u,$ \\
								&						&							&										& $\mathfrak X_7=x\partial _y-y\partial _x$ \\\hline
								&						&							& 										& $\mathfrak X_4=\partial _u,\ \mathfrak X_5 = t\partial _x-x\partial _u,$ \\
$\beta$							& $r$	   			& $\frac{1}{2}$					& $0$									& $\mathfrak X_6=t\partial _y-y\partial _u,\ \mathfrak X_7=x\partial _y-y\partial _x,$ \\		
								&						&							& 										& $\mathfrak X_8= 2(3\beta t-u  )\partial _u+4 t\partial _t+x\partial _x+y\partial _y$\\\hline
\multirow{2}{*}{$\gamma r^3$}	& \multirow{2}{*}{$r$}		& \multirow{2}{*}{$\frac{1}{2}$}		& \multirow{2}{*}{$\alpha u+\beta,\alpha\ne0$}		&$\mathfrak X_4= x\partial _y-y\partial _x,$ \\
								&						&							&	 									& $\mathfrak X_5= 4\alpha t\partial _t+\alpha x\partial _x+\alpha y\partial _y-2(\alpha u+\beta)\partial _u$ \\\hline	
  $ \zeta (\alpha u+\beta)^3,\ \alpha\ne0$			& $\gamma(\alpha u+\beta)$	& $\frac{1}{2}$						& $\delta(\alpha u+\beta), \delta\ne\gamma$		& $\mathfrak X_4 = 4\alpha t\partial _t+\alpha x\partial _x+\alpha y\partial _y-2(\alpha u+\beta)\partial _u$ \\\hline
\multirow{2}{*}{ $\delta$}				& \multirow{2}{*}{$\alpha$	}	& \multirow{2}{*}{$\gamma\ne0$}	& \multirow{2}{*}{$\beta\ne\alpha$}				& $\mathfrak X_4= \partial _u,\ \mathfrak X_5=t\partial _x-x\partial _u,$\\
								&						&							&										& $\mathfrak X_6 = 2 t\partial _y-\frac{y}{\gamma}\partial _u$ \\\hline
\multirow{2}{*}{ $\gamma$}			& \multirow{2}{*}{$\alpha$	}	& \multirow{2}{*}{$0$}	& \multirow{2}{*}{$\beta\ne\alpha$}					& $\mathfrak X_4= t\partial _x-x\partial _u,$\\
								&						&							&										& $\mathfrak X_5 = \mathcal{F}(y,t)\partial _u,\ \mathcal{F}_{  t}-\alpha \mathcal{F}_{  yy}+\mathcal{F}_{  yyyy}=0$ \\\hline
$\zeta (\alpha u+\beta)^3\footnotemark[2]$	& $\gamma(\alpha u+\beta)\footnotemark[2]$	& $\epsilon\ne\frac{1}{2}$			& $\delta(\alpha u+\beta)\footnotemark[2]$					& $\mathfrak X_4=4\alpha t\partial _t+ \alpha x\partial _x+ \alpha y\partial _y-2(\alpha u+\beta)\partial _u$\\\hline	
\multirow{2}{*}{$\delta u+\epsilon,\ \delta\ne0$}		& \multirow{2}{*}{$\alpha\footnotemark[3]$}		& \multirow{2}{*}{$\gamma\ne0,\frac{1}{2}$}	& \multirow{2}{*}{$\beta\footnotemark[3]$}	& $\mathfrak X_4=e^{t \delta }\partial _u,\ \mathfrak X_5=\frac{e^{t \delta }}{\delta }\partial _x-e^{t \delta } x\partial _u$\\
								&						&							&										& $\mathfrak X_6=\frac{2 e^{t \delta }}{\delta }\partial _y-\frac{e^{t \delta } y}{\gamma }\partial _u$\\\hline
\multirow{2}{*}{$\gamma u+\delta,\ \gamma\ne0$}	& \multirow{2}{*}{$\alpha\footnotemark[3]$}		& \multirow{2}{*}{$0$}	& \multirow{2}{*}{$\beta\footnotemark[3]$}	& $\mathfrak X_4=\frac{e^{t \gamma }}{\gamma }\partial _x-e^{t \gamma } x\partial _u$\\
								&						&							&										& $\mathfrak X_5=\mathcal{F}(y,t)\partial _u ,\ \mathcal{F}_{  t}-\alpha \mathcal{F}_{  yy}-\gamma\mathcal{F}+\mathcal{F}_{  yyyy}=0$\\\hline	
$\delta$	& $\alpha\footnotemark[3]$		& $\gamma\ne0,\frac{1}{2}$	& $\beta\footnotemark[3]$	& $\mathfrak X_4=t\partial _x-x\partial _u,\ \mathfrak X_5=2 t\partial _y-\frac{y}{\gamma }\partial _u$\\\hline	
\multirow{2}{*}{$\gamma$}	& \multirow{2}{*}{$\alpha\footnotemark[3]$}		& \multirow{2}{*}{$0$}	& \multirow{2}{*}{$\beta\footnotemark[3]$}	& $\mathfrak X_4=t\partial _x-x\partial _u$\\
								&						&							&										& $\mathfrak X_5=\mathcal{F}(y,t)\partial _u ,\ \mathcal{F}_{  t}-\alpha \mathcal{F}_{  yy}+\mathcal{F}_{  yyyy}=0$\\\hline	
\multirow{2}{*}{$\beta u + \gamma,\ \beta\ne0$}	& \multirow{2}{*}{$0$}		& \multirow{2}{*}{$\alpha\ne0,\frac{1}{2}$}	& \multirow{2}{*}{$0$}	& $\mathfrak X_4=e^{t \alpha }\partial _u,\ \mathfrak X_5=\frac{e^{t \beta }}{\beta }\partial _x-e^{t \beta } x\partial _u$\\
								&						&							&										& $\mathfrak X_6=\frac{2 e^{t \beta }}{\beta }\partial _y-\frac{e^{t \beta } y}{\alpha }\partial _u$\\\hline	
\multirow{2}{*}{$\alpha u + \beta,\ \alpha\ne0$}	& \multirow{2}{*}{$0$}		& \multirow{2}{*}{$0$}	& \multirow{2}{*}{$0$}	& $\mathfrak X_4= \frac{e^{t \alpha }}{\alpha }\partial _x-e^{t \alpha } x\partial _u$\\
								&						&							&										& $\mathfrak X_5=\mathcal{F}(y,t)\partial _u,\ \mathcal{F}_{  t}-\alpha \mathcal{F}+\mathcal{F}_{  yyyy}=0$\\\hline	
\multirow{3}{*}{$\beta$}	& \multirow{3}{*}{$0$}		& \multirow{3}{*}{$\alpha\ne0,\frac{1}{2}$}	& \multirow{3}{*}{$0$}	& $\mathfrak X_4=\partial _u,\ \mathfrak X_5=t\partial _x-x\partial _u$\\
								&						&							&										& $\mathfrak X_6=2 t\partial _y-\frac{y}{\alpha }\partial _u,$\\						
								&						&							&										& $ \mathfrak X_7=4 t\partial _t+x\partial _x+y\partial _y-2(u-3 \beta t)\partial _u$\\\hline		
\multirow{2}{*}{$\alpha$}	& \multirow{2}{*}{$0$}		& \multirow{2}{*}{$0$}	& \multirow{2}{*}{$0$}	& $\mathfrak X_4=t\partial _x-x\partial _u,\ \mathfrak X_5=4 t\partial _t+x\partial _x+y\partial _y-2 (u-3\beta t)\partial _u$\\
								&						&							&										& $\mathfrak X_6=\mathcal{F}(y,t)\partial _u,\ \mathcal{F}_{  t}-\alpha. \mathcal{F}+\mathcal F_{yyyy}=0$\\	\hline					
\end{tabular}\footnotetext[2]{$\alpha(\gamma^2+\delta^2+\zeta^2)\ne0$}\footnotetext[3]{$\alpha^2+\beta^2\ne0$}
\caption{Complete group classification for equation \eqref{gKS2}}
\label{tab:sym}
\end{center}
\end{sidewaystable}%
}


\section{Conservation Laws}

By consulting table \ref{tab:sym} one can observe that the only point symmetries admitted by the self-adjoint classification given in Section \ref{sec:sa}  are  $\mathfrak X_1=\partial _t,\ \mathfrak X_2=\partial _x,\ \mathfrak X_3=\partial _y$. The only exception is the nonlinear self-adjoint case,
$$
	u_{  t}=\frac{1}{2}u_x^2+\frac{1}{2} u_y^2+\left(\frac{u}{2}+\alpha\right)(u_{xx}+ u_{  yy})-u_{  xxxx}-2 u_{  xxyy}-u_{  yyyy}+ \delta u^2+\epsilon u+\zeta
$$
which also admits the symmetry $\mathfrak X_4=x\partial _y-y\partial _x$. The nontrivial conserved vectors for each case follow.

\begin{case}[Quasi self-adjoint]
According to Theorem \ref{thm:1} the quasi self-adjoint case is 
 $$
 	u_{t}=\alpha+g' u_{y}{}^2+g(u) u_{yy}-u_{yyyy}+\frac{u_{x}{}^2}{2}+\left(\beta+\frac{1}{2} u\right) u_{xx}-2 u_{xxyy}-u_{xxxx},
$$
with formal Lagrangian $\mathcal L = c \Delta(x,y,t,u,u_x,\dots,u_{yyyy})$. Using the formula~\eqref{cvf}, for each one of the above mentioned three point symmetries,  three trivial conserved vectors are obtained.

\end{case}
\begin{case}[Nonlinear self-adjoint]
Following Theorem~\ref{thm:2}, to each one of the five possible nonlinear self-adjoint cases formula~\eqref{cvf} is applied using each of its admitted point symmetries. The conserved vector is given in the form $(C^1,C^2,C^3)$  yielding the conservation law $D_xC^1+D_yC^2+D_tC^3=0$ for every solution of the case of equation~\eqref{gKS2} studied.
	\begin{subcase}[$r=u/2+\alpha, g=\beta u+\gamma, h=\beta, f=\delta u^2+\epsilon u +\zeta$]
		For this case equation~\eqref{gKS2} takes the form:
		 \begin{multline}\label{sc21}
		 	u_{t}=\zeta+\epsilon u+\delta u^2+\beta u_{y}{}^2+\left(\gamma+\beta u\right) u_{yy}-u_{yyyy}+\frac{u_{x}{}^2}{2}\\
			+\left(\alpha+\frac{u}{2} \right) u_{xx}-2 u_{xxyy}-u_{xxxx}.
		\end{multline}
		The new dependent variable $\upsilon=\mathcal{F}_1(x,y,t)$, where $\mathcal{F}_1(x,y,t)$ is  any solution of the system 
		\begin{gather*}
			2 \delta \mathcal{F}_1+\beta {\mathcal{F}_1}_{yy}+\frac{{\mathcal{F}_1}_{xx}}{2}=0,\\
		\epsilon \mathcal{F}_1+{\mathcal{F}_1}_{t}+\gamma {\mathcal{F}_1}_{yy}-{\mathcal{F}_1}_{yyyy}+\alpha {\mathcal{F}_1}_{xx}-2 {\mathcal{F}_1}_{xxyy}-{\mathcal{F}_1}_{xxxx}=0.
		\end{gather*}
		\begin{pt}
			\begin{align*}
				C^1 &= \frac{1}{12} \left(4 (1-6 \beta ) {\mathcal{F}_1}_{yyt} u_{x}+3 u^2 {\mathcal{F}_1}_{xt}-8 {\mathcal{F}_1}_{yt} u_{xy}+8 u_{y} {\mathcal{F}_1}_{xyt}\right.\\
				&\left.\qquad\quad-6 {\mathcal{F}_1}_{t} \left((2 \alpha +8 \delta +u) u_{x}-2 \left(u_{xyy}+u_{xxx}\right)\right)-4 {\mathcal{F}_1}_{xt} \left(u_{yy}+3 u_{xx}\right)\right)\\
				&\qquad\quad+ u \left((\alpha +4 \delta ) {\mathcal{F}_1}_{xt}-{\mathcal{F}_1}_{xyyt}\right)+2 \beta  u {\mathcal{F}_1}_{xyyt},\\
				C^2 &= \frac{1}{6} \left(3 \beta  u^2 {\mathcal{F}_1}_{yt}+{\mathcal{F}_1}_{t} \left(-2 (3 \gamma +4 \delta +3 \beta  u) u_{y}+6 \left(u_{yyy}+u_{xxy}\right)\right)\right.\\
				&\qquad\left.-2 \left((-3+2 \beta ) u_{y} {\mathcal{F}_1}_{yyt}+2 {\mathcal{F}_1}_{xt} u_{xy}-2 u_{x} {\mathcal{F}_1}_{xyt}+{\mathcal{F}_1}_{yt} \left(3 u_{yy}+u_{xx}\right)\right)\right)\\
				&\qquad+ u \left((\gamma +4 \delta ) {\mathcal{F}_1}_{yt}+(-1+2 \beta ) {\mathcal{F}_1}_{yyyt}\right),\\
				C^3 &= -\zeta  \mathcal{F}_1(x,y,t)+u {\mathcal{F}_1}_{ t}.
			\end{align*}		
		\end{pt}	
		This conserved vector will be nontrivial if and only if ${\mathcal{F}_1}_{ t}\ne0$.
		\begin{px}
			\begin{align*}
				C^1 &= \frac{1}{6} \left(3 u_{, x} \left(2 (1-2 \beta ) {\mathcal{F}_1}_{, xyy}-(2 \alpha +8 \delta +u) {\mathcal{F}_1}_{, x}\right)-2 {\mathcal{F}_1}_{, y} u_{, xxy}\right.\\
				&\left.+2 \left((6 \beta-1 ) {\mathcal{F}_1}_{, yy} u_{, xx}+{\mathcal{F}_1}_{, x} \left(u_{, xyy}+3 u_{, xxx}\right)-2 u_{, xy} {\mathcal{F}_1}_{, xy}\right)\right)- {\mathcal{F}_1}_{, y} u_{, yyy},\\
				&- \mathcal{F}_1(x,y,t) \left(\zeta -u_{, t}+\beta  u_{, y}{}^2+\gamma  u_{, yy}+u \left(\epsilon +\delta  u+\beta  u_{, yy}\right)-4 \delta  u_{, xx}\right)\\
				C^2 &= \frac{1}{3} \left((-3+2 \beta ) {\mathcal{F}_1}_{yy} u_{xy}+{\mathcal{F}_1}(x,y,t)\left(3 \beta  u_{y} u_{x}+(3 \gamma +4 \delta +3 \beta  u) u_{xy}\right)\right.\\
				&\qquad\quad\left.-2 {\mathcal{F}_1}_{xy} u_{xx}+{\mathcal{F}_1}_{x} \left(3 u_{yyy}+5 u_{xxy}\right)+{\mathcal{F}_1}_{y}u_{xxx}\right)+ (1-2 \beta ) {\mathcal{F}_1}_{yyy} u_{x}\\
				&\qquad\quad+{\mathcal{F}_1}_{y} \left(u_{xyy}- (\gamma +4 \delta +\beta  u) u_{x} \right),\\
				C^3 &= -\mathcal{F}_1(x,y,t) u_{ x}.
			\end{align*}		
		\end{px}
		 This conserved vector will be nontrivial if and only if ${\mathcal{F}_1}_{ x}\ne0$.
		\begin{py}
			\begin{align*}
				C^1 &= \frac{1}{6} \left(2 \left(-2 u_{xy} {\mathcal{F}_1}_{xy}+(-1+6 \beta ) {\mathcal{F}_1}_{yy} u_{xx}+{\mathcal{F}_1}_{x} \left(u_{xyy}+3 u_{xxx}\right)\right)\right.\\
				&\qquad\left.-3 (2 \alpha +8 \delta +u) u_{x} {\mathcal{F}_1}_{x}-2 {\mathcal{F}_1}_{y} \left(3 u_{yyy}+u_{xxy}\right)\right)+ (1-2 \beta ) {\mathcal{F}_1}_{xyy} u_{x}\\
				& \qquad- {\mathcal{F}_1}(x,y,t) \left(\zeta -u_{t}+\beta  u_{y}{}^2+\gamma  u_{yy}+u \left(\epsilon +\delta  u+\beta  u_{yy}\right)-4 \delta  u_{xx}\right),\\
				C^2 &= \frac{1}{3} \left((2 \beta -3) {\mathcal{F}_1}_{yy} u_{xy}+(3 \gamma +4 \delta +3 \beta  u) u_{xy}{\mathcal{F}_1}(x,y,t) -2 {\mathcal{F}_1}_{xy} u_{xx}\right.\\
				&\qquad\quad\left.+5 u_{xxy}{\mathcal{F}_1}_{x}+u_{xxx}{\mathcal{F}_1}_{y}\right)+(1-2 \beta ) {\mathcal{F}_1}_{yyy} u_{x}\\
				&\qquad\quad+ \beta  u_{y} u_{x}{\mathcal{F}_1}(x,y,t)+ u_{yyy}{\mathcal{F}_1}_{x}+{\mathcal{F}_1}_{y} \left(u_{xyy}- (\gamma +4 \delta +\beta  u) u_{x}\right),\\
				C^3 &= -{\mathcal{F}_1}(x,y,t) u_{x}.
			\end{align*}
		\end{py}
		This conserved vector will be nontrivial if and only if ${\mathcal{F}_1}_{y}\ne0$.
		\begin{pxpy} 
			If, in addition, $\beta=1/2$ and $\gamma=\alpha$ then equation~\eqref{sc21} admits also the symmetry $x\partial _y-y\partial _x$. Employing it we have:
			\begin{align*}
				C^1 &= \frac{1}{6} \left(2 x u_{yyy} {\mathcal{F}_1}_{x}+3 y u u_{x} {\mathcal{F}_1}_{x}+u_{y} \left(4 {\mathcal{F}_1}_{yy}-3 x (2 \alpha +8 \delta +u) {\mathcal{F}_1}_{x}\right)\right.\\
				&\qquad\,+4 x {\mathcal{F}_1}_{yy} u_{xy}+8 {\mathcal{F}_1}_{x} u_{xy}-4 x u_{yy} {\mathcal{F}_1}_{xy}+4 u_{x} {\mathcal{F}_1}_{xy}+4 y u_{xy} {\mathcal{F}_1}_{xy}\\
				&\qquad\,-2 y {\mathcal{F}_1}_{x} u_{xyy}-4 y {\mathcal{F}_1}_{yy} u_{xx}+2 {\mathcal{F}_1}_{y} \left(2 u_{yy}+2 x u_{xyy}-2 u_{xx}+y u_{xxy}\right)\\
				&\qquad\,\left.+3 {\mathcal{F}_1}(x,y,t) \left(y u_{y}{}^2+x u_{y} u_{x}+u \left(u_{y}+y u_{yy}+x u_{xy}\right)\right)\right)+ y \alpha  u_{x} {\mathcal{F}_1}_{x}\\
				&\qquad\,+4 y \delta  u_{x} {\mathcal{F}_1}_{x}+x {\mathcal{F}_1}_{x} u_{xxy}-y {\mathcal{F}_1}_{x} u_{xxx}+ {\mathcal{F}_1}_{y} \left( y u_{yyy}+ x u_{xxx}\right)\\
				&\qquad\,+{\mathcal{F}_1}(x,y,t) \left(y \zeta +y \delta  u^2-y u_{t}+(\alpha +4 \delta ) u_{y}+ y \alpha  u_{yy}+ x \alpha  u_{xy}\right.\\
				&\qquad\,\left.+4 x \delta  u_{xy}+\epsilon y u - x u_{xyyy}-4 y \delta  u_{xx}\right),\\
				C^2 &= \frac{1}{6} \left(-4 x u_{yy} {\mathcal{F}_1}_{yy}+3 y u {\mathcal{F}_1}_{y} u_{x}+4 {\mathcal{F}_1}_{yy} u_{x}+4 u_{yy} {\mathcal{F}_1}_{x}-8 {\mathcal{F}_1}_{y} u_{xy}\right.\\
				&\qquad\,+4 y {\mathcal{F}_1}_{yy} u_{xy}-u_{y} \left(3 x (2 \alpha +8 \delta +u) {\mathcal{F}_1}_{y}+4 {\mathcal{F}_1}_{xy}\right)-4 x u_{xy} {\mathcal{F}_1}_{xy}\\
				&\qquad\,+4 x {\mathcal{F}_1}_{x} u_{xyy}-4 {\mathcal{F}_1}_{x} u_{xx}+4 y {\mathcal{F}_1}_{xy} u_{xx}+2 x {\mathcal{F}_1}_{y} u_{xxy}-10 y {\mathcal{F}_1}_{x} u_{xxy}\\
				&\qquad\,-2y {\mathcal{F}_1}_{y} u_{xxx}-{\mathcal{F}_1}(x,y,t) \left(8 \delta  u_{x}-8 x \delta  u_{yy}+3 y u_{y} u_{x}+3 x u_{x}{}^2\right.\\
				&\qquad\,\left.\left.+8 y \delta  u_{xy}+3 u \left(u_{x}+y u_{xy}+x u_{xx}\right)\right)\right)+x {\mathcal{F}_1}_{y} u_{yyy}+y \alpha  {\mathcal{F}_1}_{y} u_{x}\\
				&\qquad\,+4 y \delta  {\mathcal{F}_1}_{y} u_{x}-y u_{yyy} {\mathcal{F}_1}_{x}- y {\mathcal{F}_1}_{y} u_{xyy}- x {\mathcal{F}_1}_{x}u_{xxx}\\
				&\qquad\,-{\mathcal{F}_1}(x,y,t) \left(x \zeta + x \delta  u^2- x u_{t}+ \alpha  u_{x}+ y \alpha  u_{xy}- u_{xyy}+ x \alpha  u_{xx}\right.\\
				&\qquad\,\left.+  \epsilon x u- x u_{xxyy}\right),\\
				C^3 &= \mathcal{F}_0(x,y,t) \left(y u_{x}-x u_{y}\right).
			\end{align*}
		\end{pxpy}	
		This conserved vector will be nontrivial if and only if ${\mathcal{F}_1}\ne \mathcal{F}(x^2+y^2,t)$.		
	\end{subcase}
	\begin{subcase}[$r=u/2+\delta,  h=g^\prime, f=\frac{\alpha u^2}{2}+\beta u +\gamma + c \int\!\! g(u)\,du,\ \alpha,c,g^{\prime\prime}\ne0$]
		For this case equation~\eqref{gKS2} takes the form:
		 \begin{multline*}
		 	u_{t}=\frac{\alpha u^2}{2}+\beta u+\gamma+\boldsymbol{c}\int\!\! g(u)\,du+g^{\prime} u_{y}{}^2+ g(u) u_{yy}-u_{yyyy}+\frac{u_{x}^2}{2}\\+\left(\frac{u}{2}+\delta\right) u_{xx}
			-2 u_{xxyy}-u_{xxxx}.
		\end{multline*}
		The new dependent variable is
		\begin{multline*}
			\upsilon=e^{t \left(4 \alpha ^2-\beta +\boldsymbol{c}{}^2+2 \alpha  \left(\delta +2 \boldsymbol{c}\right)\right)}\left(\cos(\sqrt{2\alpha }x) \left(\boldsymbol{c}_1 \cos(\sqrt{\boldsymbol{c}}y)+\boldsymbol{c}_3 \sin(\sqrt{\boldsymbol{c}}y)\right)+\right.\\
				      \left.\sin(\sqrt{2\alpha }x) \left(\boldsymbol{c}_2 \cos(\sqrt{\boldsymbol{c}}y)+\boldsymbol{c}_4 \sin(\sqrt{\boldsymbol{c}}y)\right)\right).
		\end{multline*}
		\begin{pt}
			\begin{align*}
				C^1 &= \frac{1}{4}e^{\mathcal{D} t}  \mathcal{D}  \left(\mathcal{A}\sqrt{2\alpha }  \left(8 \alpha +4 \delta +8 \boldsymbol{c}+u\right)u-2 \mathcal{B}  \left(4 \alpha +2 \delta +4 \boldsymbol{c}+u\right) u_{x}\right.\\
				  	&\qquad\qquad\left.-4 \mathcal{A} \sqrt{2\alpha } \mathcal{A}   u_{xx}+4 \mathcal{B} u_{xxx}\right),\\
				C^2 &= e^{\mathcal{D} t} \mathcal{D} \left(\mathcal{C}  \sqrt{\boldsymbol{c}} \left(\int\!\! g(u)\,du+\boldsymbol{c} u\right)-\mathcal{B} \left(\boldsymbol{c}+g(u)\right) u_{y}-\mathcal{C}  \sqrt{\boldsymbol{c}} u_{yy}\right.\\
					&\qquad\quad\left.+\mathcal{B}\mathcal{D} u_{yyy}-2 \mathcal{C} \mathcal{D} \sqrt{\boldsymbol{c}} u_{xx}+2 \mathcal{B} \mathcal{D} u_{xxy}\right),\\
				C^3 &= -e^{\mathcal{D} t}\mathcal{B} (\gamma -\mathcal{D} u).
			\end{align*}
		\end{pt}	
		This conserved vector will be nontrivial if and only if $\mathcal{D}\ne0$.	
		\begin{px}
			\begin{align*}
				C^1 &= \frac{1}{2} e^{\mathcal{D} t}\left(4 \mathcal{B} \alpha  u_{xx}+2  \sqrt{2\alpha }\mathcal{A} u_{xxx}-2 \mathcal{B} \gamma -4 \mathcal{B} \alpha  \left(2 \alpha +\delta +2 \boldsymbol{c}\right) u-\mathcal{B} \alpha  u^2\right.\\
					&\qquad\qquad\left. -\sqrt{2\alpha }\mathcal{A} \left(4 \alpha +2 \delta +4 \boldsymbol{c}+u\right) u_{x}\right),\\
				C^2 &= \sqrt{2\alpha} e^{\mathcal{D} t}\left(\mathcal{E} \sqrt{\boldsymbol{c}} \int\!\! g(u)\,du+\mathcal{E} \boldsymbol{c}{}^{3/2} u-(\mathcal{A} \boldsymbol{c}+\mathcal{A} g(u) )u_{y}-\mathcal{E} \sqrt{\boldsymbol{c}} u_{yy}\right.\\
					&\qquad\qquad\quad\left.+\mathcal{A} u_{yyy}-2 \mathcal{E} \sqrt{\boldsymbol{c}} u_{xx}+2 \mathcal{A} u_{xxy}\right),\\
				C^3 &= \sqrt{2\alpha }e^{\mathcal{D} t} \mathcal{A}  u.
			\end{align*}		
		\end{px}								
		\begin{py}
			\begin{align*}
				C^1 &=  \frac{1}{4}\sqrt{\boldsymbol{c}}e^{\mathcal{D} t}\left(4 \mathcal{C} u_{xxx}+\sqrt{2} \mathcal{E} \sqrt{\alpha } \left(u \left(8 \alpha +4 \delta +8 \boldsymbol{c}+u\right)-4 u_{xx}\right)\right.\\
					&\left.\qquad\qquad\qquad-2 \mathcal{C} \left(4 \alpha +2 \delta +4 \boldsymbol{c}+u\right) u_{x}\right),\\
				C^2 &=  -e^{\mathcal{D} t}\left(\mathcal{B} \gamma +\mathcal{B} \boldsymbol{c} \int\!\! g(u)\,du+\mathcal{B} \boldsymbol{c}{}^2 u+\mathcal{C} \sqrt{\boldsymbol{c}} \left(\boldsymbol{c}+g(u)\right) u_{y}\right.\\
					&\left.\qquad\qquad-\sqrt{\boldsymbol{c}} \left(\mathcal{B} \sqrt{\boldsymbol{c}} \left(u_{yy}+2 u_{xx}\right)+\mathcal{C} \left(u_{yyy}+2 u_{xxy}\right)\right)\right),\\
				C^3 &=   \mathcal{C} \sqrt{\boldsymbol{c}}e^{\mathcal{D} t}u.
			\end{align*}	
		\end{py}
		In the above formulae
		\begin{align*}
			\mathcal{\mathcal{A}}&=\sin(\sqrt{2\alpha }x) \left(\boldsymbol{c}_1 \cos(\sqrt{\boldsymbol{c}}y)+\boldsymbol{c}_3 \sin(\sqrt{\boldsymbol{c}}y)\right)\\
			&\qquad\qquad\qquad\quad-\cos(\sqrt{2\alpha }x) \left(\boldsymbol{c}_2 \cos(\sqrt{\boldsymbol{c}}y)+\boldsymbol{c}_4 \sin(\sqrt{\boldsymbol{c}}y)\right),\\
			\mathcal{\mathcal{B}}&=\cos(\sqrt{2\alpha }x) \left(\boldsymbol{c}_1 \cos(\sqrt{\boldsymbol{c}}y)+\boldsymbol{c}_3 \sin(\sqrt{\boldsymbol{c}}y)\right)\\
			&\qquad\qquad\qquad\quad+\sin(\sqrt{2\alpha }x) \left(\boldsymbol{c}_2 \cos(\sqrt{\boldsymbol{c}}y)+\boldsymbol{c}_4 \sin(\sqrt{\boldsymbol{c}}y)\right),\\
			\mathcal{\mathcal{C}}&=\cos(\sqrt{2\alpha }x) \left(\boldsymbol{c}_3 \cos(\sqrt{\boldsymbol{c}}y)-\boldsymbol{c}_1 \sin(\sqrt{\boldsymbol{c}}y)\right)\\
			&\qquad\qquad\qquad\quad+\sin(\sqrt{2\alpha }x) \left(\boldsymbol{c}_4 \cos(\sqrt{\boldsymbol{c}}y)-\boldsymbol{c}_2 \sin(\sqrt{\boldsymbol{c}}y)\right),\\
			\mathcal{\mathcal{D}}&=2 \alpha  (2 \alpha +\delta )+4 \alpha  \boldsymbol{c}+\boldsymbol{c}{}^2-\beta\\
			\text{and} \\
			\mathcal{E}&=\sin(\sqrt{2\alpha }x) \left(\boldsymbol{c}_1 \sin(\sqrt{\boldsymbol{c}}y)-\boldsymbol{c}_3 \cos(\sqrt{\boldsymbol{c}}y)\right)\\
			&\qquad\qquad\qquad\quad+\cos(\sqrt{2\alpha }x) \left(\boldsymbol{c}_4 \cos(\sqrt{\boldsymbol{c}}y)-\boldsymbol{c}_2 \sin(\sqrt{\boldsymbol{c}}y)\right).
		\end{align*}						
	\end{subcase}
	\begin{subcase}[$r=u/2+\gamma,  h=g^\prime, f=\alpha  u +\beta + c \int\!\! g(u)\,du,\ c,g^{\prime\prime}\ne0$]
		For this case equation~\eqref{gKS2} assumes the form:
		 \begin{multline*}
		 	u_{t}=\beta  u +\gamma + c \int\!\! g(u)\,du+g^{\prime} u_{y}{}^2+g(u) u_{yy}-u_{yyyy}+\frac{u_{x}^2}{2}+\frac{2\alpha+u}{2} u_{xx}-2 u_{xxyy}\\
			-u_{xxxx}.
		\end{multline*}
		The new dependent variable $\upsilon$ is,
		$$
			\upsilon= e^{  \left(\boldsymbol{c}^2-\alpha\right)t} \left(\left(\boldsymbol{c}_1+ \boldsymbol{c}_2x\right) \cos(\sqrt{\boldsymbol{c}}y)+\left(\boldsymbol{c}_3+ \boldsymbol{c}_4x\right) \sin(\sqrt{\boldsymbol{c}}y)\right).
		$$
		\begin{pt}
			\begin{align*}
				C^1 &= \frac{1}{4}e^{\mathcal{D} t}\mathcal{D}  \left(\mathcal{B} u \left(4 \gamma +8 \boldsymbol{c}+u\right)-2 \mathcal{A} \left(2 \gamma +4 \boldsymbol{c}+u\right) u_{x}-4 \mathcal{B} u_{xx}+4 \mathcal{A} u_{xxx}\right),\\
				C^2 &= e^{\mathcal{D} t}\mathcal{D}\left(\mathcal{C} \sqrt{\boldsymbol{c}} \int\!\! g(u)\,du+\mathcal{C} \boldsymbol{c}{}^{3/2} u-\mathcal{A} \boldsymbol{c} u_{y}-\mathcal{A} g(u) u_{y}-\mathcal{C} \sqrt{\boldsymbol{c}} u_{yy}+\mathcal{A} u_{yyy}\right.\\
					&\qquad\qquad\left.-2 \mathcal{C} \sqrt{\boldsymbol{c}} u_{xx}+2 \mathcal{A} u_{xxy}\right), \\
				C^3 &= -e^{\mathcal{D} t}\mathcal{A} (\beta -\mathcal{D} u).
			\end{align*}		
		\end{pt}	
		This conserved vector will be nontrivial if and only if $\mathcal{D}\ne0$.	
		\begin{px}
			\begin{align*}
				C^1 &=\frac{1}{2} e^{\mathcal{D} t} \mathcal{B}  \left(u \left(2 \mathcal{D} x-u_{x}\right)+2 \left(x u_{t}- \beta x -\gamma  u_{x}+2 u_{xyy}+u_{xxx}\right)\right),\\
				C^2 &= e^{\mathcal{D} t} \left(\mathcal{E} \sqrt{\boldsymbol{c}} \left(\int\!\! g(u)\,du+\boldsymbol{c} u\right)-\mathcal{B} \left(\boldsymbol{c}+g(u)\right) u_{y}-\mathcal{E} \sqrt{\boldsymbol{c}} u_{yy}+\mathcal{B} u_{yyy}\right),\\
				C^3 &= -e^{\mathcal{D} t}\mathcal{B} x u_{x}.
			\end{align*}		
		\end{px}
		This conserved vector will be nontrivial if and only if $\boldsymbol{c}_2^2+\boldsymbol{c}_4^2\ne0$.	
		\begin{py}
			\begin{align*}
				C^1 &= \frac{1}{4}e^{\mathcal{D} t}\sqrt{\boldsymbol{c}} \left(2  \beta x \left(2 \boldsymbol{c}_1+ \boldsymbol{c}_2x\right) \sin(\sqrt{\boldsymbol{c}}y)+\mathcal{E} u^2-4 \mathcal{E} \left(2 u_{yy}+u_{xx}\right)\right.\\
					&\qquad\qquad\qquad -2 \beta  x\left(2 \boldsymbol{c}_3+ \boldsymbol{c}_4x\right) \cos(\sqrt{\boldsymbol{c}}y)+u \left(4 \mathcal{E} \gamma -2 \mathcal{C} u_{x}\right)\\
					&\qquad\qquad\qquad\left.+4 \mathcal{C} \left(u_{xxx}-\gamma  u_{x}+2 u_{xyy}\right)\right),\\
				C^2 &= \sqrt{\boldsymbol{c}}e^{\mathcal{D} t}\left(\mathcal{A} \sqrt{\boldsymbol{c}} u_{yy}+\mathcal{C} u_{yyy}-\mathcal{A} \sqrt{\boldsymbol{c}} \left(\int\!\! g(u)\,du+\boldsymbol{c} u\right)\right.\\
				&\qquad\qquad\quad\left.-\mathcal{C} \left(\boldsymbol{c}+g(u)\right) u_{y}\right),\\
				C^3 &= \mathcal{C} \sqrt{\boldsymbol{c}} e^{\mathcal{D} t}.
			\end{align*}		
		\end{py}	
		For the above conserved vectors we have used the following notation:
		\begin{align*}
			\mathcal{A}&=\left(\boldsymbol{c}_1+ \boldsymbol{c}_2x\right)\cos(\sqrt{\boldsymbol{c}} y)+\left(\boldsymbol{c}_3+ \boldsymbol{c}_4x\right)\sin(\sqrt{\boldsymbol{c}} y),\\
			\mathcal{B}&=\boldsymbol{c}_2\cos(\sqrt{\boldsymbol{c}} y)+\boldsymbol{c}_4\sin(\sqrt{\boldsymbol{c}} y),\\
			\mathcal{C}&=\left(\boldsymbol{c}_3+ \boldsymbol{c}_4x\right)\cos(\sqrt{\boldsymbol{c}} y)-\left(\boldsymbol{c}_1+ \boldsymbol{c}_2x\right)\sin(\sqrt{\boldsymbol{c}} y),\\
			\mathcal{D}&= \boldsymbol{c}{}^2-\alpha\\
			\text{and} \\
			\mathcal{E}&=\boldsymbol{c}_4\cos(\sqrt{\boldsymbol{c}} y)-\boldsymbol{c}_2\sin(\sqrt{\boldsymbol{c}} y).
		\end{align*}								
	\end{subcase}
	\begin{subcase}[$r=u/2+\delta,  h=g^\prime, f=\alpha u^2+\beta u +\gamma ,\ \alpha,g^{\prime\prime}\ne0$]
		For this case equation~\eqref{gKS2} takes the form:
		 \begin{multline*}
		 	u_{t}=u_{t}=\alpha  u^2 +\beta  u+\gamma+g^\prime u_{y}{}^2+g(u) u_{yy}-u_{yyyy}+\frac{u_{x}{}^2}{2}+\left(\delta +\frac{1}{2} u\right) u_{xx}\\
			-2 u_{xxyy}-u_{xxxx}.
		\end{multline*}
		The new dependent variable $\upsilon$ is,
		$$
			\upsilon = e^{\left(16 \alpha ^2-\beta +4 \alpha  \delta \right)t} \left(\left(\boldsymbol{c}_1y +\boldsymbol{c}_2\right) \cos(2 \sqrt{\alpha }x)+\left(\boldsymbol{c}_3 y+\boldsymbol{c}_4\right) \sin(2 \sqrt{\alpha }x)\right).
		$$
		\begin{pt}
			\begin{align*}
				C^1 &= \frac{1}{2}e^{\mathcal{D} t}\mathcal{D}\left(\mathcal{B} \sqrt{\alpha } u^2-2 \left(\mathcal{A} (4 \alpha +\delta ) u_{x}+2 \mathcal{B} \sqrt{\alpha } u_{xx}-\mathcal{A} u_{xxx}\right)\right.\\
						&\qquad\qquad\quad\left.+ \left(4 \mathcal{B} \sqrt{\alpha } (4 \alpha +\delta )u-\mathcal{A} u_{x}\right)\right). \\
				C^2 &= -e^{\mathcal{D} t}\left(g(u) \left(\mathcal{C} u_{t}+\mathcal{A} \mathcal{D} u_{y}\right)+\mathcal{D} \left(\mathcal{C} u_{yy}-\mathcal{A} u_{yyy}+2 \left(\mathcal{C} u_{xx}-\mathcal{A} u_{xxy}\right)\right)\right).\\
				C^3 &= e^{\mathcal{D} t}\left(\mathcal{A} \mathcal{D} u+\mathcal{C} g(u) u_{y}-\mathcal{A} \gamma\right);
			\end{align*}		
		\end{pt}
		this conserved vector will be nontrivial if and only if $\mathcal{D}\ne0$.	
		\begin{px}
			\begin{align*}
				C^1 &= e^{\mathcal{D} t}\left(\mathcal{C} g(u) u_{y}-\mathcal{A} \alpha  u^2-u \left(4 \mathcal{A} \alpha  (4 \alpha +\delta )+\mathcal{B} \sqrt{\alpha } u_{x}\right)\right.\\
					&\qquad\qquad\left. -2 \sqrt{\alpha } \left(\mathcal{B} (4 \alpha +\delta ) u_{x}-2 \mathcal{A} \sqrt{\alpha } u_{xx}-\mathcal{B} u_{xxx}\right)\right),\\
				C^2 &=  e^{\mathcal{D} t}\left(\sqrt{\alpha }\left(2 \mathcal{B} u_{yyy}+\sin(2 \sqrt{\alpha }x) \left( \gamma  \left(\boldsymbol{c}_1y+2 \boldsymbol{c}_2\right)y-4 \left(\boldsymbol{c}_1y+\boldsymbol{c}_2\right) u_{xxy}  \right.\right.\right.\\
					&\qquad\qquad\left. +2 \boldsymbol{c}_1 \left(u_{yy}+2 u_{xx}\right)\right)-\cos(2  \sqrt{\alpha }x) \left( 2 \boldsymbol{c}_3 \left(u_{yy}+2 u_{xx}\right)\right.\\
					&\qquad\qquad\left.\left.+\gamma  \left(\boldsymbol{c}_3y+2 \boldsymbol{c}_4\right)y-4 \left(\boldsymbol{c}_3y+\boldsymbol{c}_4\right) u_{xxy}\right)\right)\\
					&\qquad\qquad\left.-g(u) \left(2 \mathcal{B} \sqrt{\alpha } u_{y}+\mathcal{C} u_{x}\right)\right),\\
				C^3 &= 2e^{\mathcal{D} t}\mathcal{B} \sqrt{\alpha } u.
			\end{align*}		
		\end{px}
		\begin{py}
			\begin{align*}
				C^1 &= \frac{1}{2} e^{\mathcal{D} t}\left(\mathcal{E} \sqrt{\alpha } u^2+ \left(4 \mathcal{E} \sqrt{\alpha } (4 \alpha +\delta )-\mathcal{C} u_{x}\right)u\right.\\
							&\qquad\qquad\left.-2 \left(\mathcal{C} (4 \alpha +\delta ) u_{x}+2 \mathcal{E} \sqrt{\alpha } u_{xx}-\mathcal{C} u_{xxx}\right)\right),\\
				C^2 &= \mathcal{C} e^{\mathcal{D} t}\left(u_{yyy}+2 u_{xxy}-\gamma y -g(u) u_{y}\right),\\
				C^3 &= e^{\mathcal{D} t}\mathcal{C}  u;
			\end{align*}		
		\end{py}
		
		where 
		\begin{align*}
			\mathcal{A}&=\left(\boldsymbol{c}_1y+\boldsymbol{c}_2\right) \cos(2 \sqrt{\alpha}x)+\left(\boldsymbol{c}_3 y+\boldsymbol{c}_4\right) \sin(2 \sqrt{\alpha }x),\\
			\mathcal{B}&=\left(\boldsymbol{c}_3 y +\boldsymbol{c}_4\right) \cos(2\sqrt{\alpha }x)-\left(\boldsymbol{c}_1y+\boldsymbol{c}_2\right) \sin(2 \sqrt{\alpha }x),\\
			\mathcal{C}&=\boldsymbol{c}_1 \cos(2\sqrt{\alpha }x)+\boldsymbol{c}_3 \sin(2 \sqrt{\alpha }x),\\
			\mathcal{D}&=16 \alpha ^2-\beta +4 \alpha  \delta\\
			\text{and} \\
			\mathcal{E}&= \boldsymbol{c}_3 \cos(2\sqrt{\alpha }x)-\boldsymbol{c}_1 \sin(2 \sqrt{\alpha }x).
		\end{align*}	
		This conserved vector will be nontrivial if and only if $\boldsymbol{c}_1^2+\boldsymbol{c}_3^2\ne0$.	
	\end{subcase}	
	\begin{subcase}[$r=u/2+\gamma,  h=g^\prime, f=\alpha  u +\beta,\ g^{\prime\prime}\ne0$]
		For this case equation~\eqref{gKS2} assumes the form:
		 \begin{multline*}
		 	u_{t}=u_{t}=\alpha u +\beta +g^\prime u_{y}{}^2+g(u) u_{yy}-u_{yyyy}+\frac{u_{x}{}^2}{2}+\left(\gamma +\frac{1}{2} u\right) u_{xx}-2 u_{xxyy}\\
			-u_{xxxx}.
		\end{multline*}
		The new dependent variable $\upsilon$ is,
		$$
			\upsilon = 	e^{-\alpha t} \left(\left(\boldsymbol{c}_1+ \boldsymbol{c}_2x\right)y+\boldsymbol{c}_3+ \boldsymbol{c}_4x\right).	
		$$
		\begin{pt}
			\begin{align*}
				C^1 &= \frac{\alpha}{4} e^{-\alpha t}\left(2\left(\mathcal{A}u u_{x}-2 \gamma\mathcal{B}\right)+4\gamma\mathcal{A} u_{x}+4\mathcal{B} u_{xx}-\mathcal{B} u^2-4\mathcal{A} u_{xxx}\right),\\
				C^2 &= e^{-\alpha t} \left(g(u) \left(\alpha  \mathcal{A} u_{y}-\mathcal{C} u_{t}\right)+\alpha  \left(\mathcal{C} u_{yy}-\mathcal{A} u_{yyy}+2\mathcal{C} u_{xx}-2 \mathcal{A} u_{xxy}\right)\right),\\
				C^3 &= e^{-\alpha t}\left(\mathcal{C} g(u) u_{y}-\mathcal{A} (\beta +\alpha  u)\right).
			\end{align*}		
		\end{pt}		
		This conserved vector will be nontrivial if and only if $\alpha\ne0$.		
		\begin{px}
			\begin{align*}
				C^1 &= \frac{1}{2}e^{-\alpha t}\left(2 \boldsymbol{c}_2 x g(u)u_{y}-\mathcal{B}\left(2 \beta x+(2 \gamma +u) u_{x}-4 u_{xyy}-2 u_{xxx}\right)\right),\\
				C^2 &= e^{-\alpha t}\left(\mathcal{B} u_{yyy}-\boldsymbol{c}_2 u_{yy}-g(u) \left(\mathcal{B} u_{y}+\boldsymbol{c}_2x u_{x}\right)\right),\\
				C^3 &= e^{-\alpha t}\mathcal{B} u.
			\end{align*}		
		\end{px}		
		This conserved vector will be nontrivial if and only if $\boldsymbol{c}_2^2+\boldsymbol{c}_4^2\ne0$.			
		\begin{py}
			\begin{align*}
				C^1 &=  \frac{1}{4}e^{-\alpha t}\left(\boldsymbol{c}_2 u^2-4 \gamma  \mathcal{C} u_{x}+2 u \left(2 \gamma  \boldsymbol{c}_2-\mathcal{C}u_{x}\right)-4 \boldsymbol{c}_2 u_{xx}+4 \mathcal{C} u_{xxx}\right),\\
				C^2 &=  e^{-\alpha t}\mathcal{C}\left(u_{yyy}+2 u_{xxy}- \beta y -g(u) u_{y}\right),\\
				C^3 &=  e^{-\alpha t}\mathcal{C} u.
			\end{align*}		
		\end{py}
		This conserved vector will be nontrivial if and only if $\boldsymbol{c}_1^2+\boldsymbol{c}_2^2\ne0$.
		
		In the last three conserved vectors
		\begin{align*}
			\mathcal{A}&=\left(\boldsymbol{c}_1+ \boldsymbol{c}_2x\right)y+\boldsymbol{c}_3+ \boldsymbol{c}_4x,\\
			\mathcal{B}&=\boldsymbol{c}_2 y+\boldsymbol{c}_4,\\
			\text{and} \\
			\mathcal{C}&= \boldsymbol{c}_1+\boldsymbol{c}_2x.
		\end{align*}		
	\end{subcase}			
\end{case}

\section{Conclusions}
In the present work a generalization of the anisotropic \KSII\ was studied under the prism of the modern group analysis. Specifically, two distinct classifications were performed; one with respect to the Lie point symmetries and a second with respect to the property of self-adjointness. The wealth of information obtained by those two classifications not only shed light to the structure of the generalization studied, by highlighting the interesting subcases, but provides us also with ways for attaining analytical nontrivial solutions.  A fact that will be the subject of a future work.

\section*{Acknowledgements}

Yuri Bozhkov would like to thank CNPq, Brazil, for partial financial support. Stylianos Dimas is grateful to FAPESP (Proc. \#2011/05855-9)  for the financial support and IMECC-UNICAMP for their gracious hospitality. 

\appendix

\section{The Determining Equations}\label{app:det}

$$
 	\xi ^3{}_{x}=0,
$$
$$
	 \xi ^3{}_{y}=0 ,
$$
$$
 	\xi ^3{}_{u}=0 ,
$$
$$
 	\xi ^3{}_{xy}=0 ,
$$
$$
	 \xi ^3{}_{xu}=0 ,
$$
$$
 	\xi ^3{}_{yu}=0 ,
$$
$$
	 \xi ^3{}_{uu}=0 ,
$$
$$
	 \xi ^3{}_{xuu}=0 ,
$$
$$
 	\xi ^3{}_{xyu}=0 ,
$$
$$
 	\xi ^3{}_{yuu}=0 ,
$$
$$
	 \xi ^3{}_{uuu}=0 ,
$$
$$
 	\xi ^3{}_{xuuu}=0 ,
$$
$$
 	\xi ^3{}_{xyuu}=0 ,
$$
$$
 	\xi ^3{}_{yuuu}=0 ,
$$
$$
	 \xi ^3{}_{uuuu}=0 ,
$$
$$
 	g(u) \xi ^3{}_{u}=0 ,
$$
$$
 	r(u) \xi ^3{}_{u}=0 ,
$$
$$
 	3 \xi ^3{}_{yy}+\xi ^3{}_{xx}=0 ,
$$
$$
 	\xi ^3{}_{yy}+3 \xi ^3{}_{xx}=0 ,
$$
$$
 	7\xi ^3{}_{yy}+5 \xi ^3{}_{xx}=0 ,
$$
$$
 	5 \xi ^3{}_{yy}+7 \xi ^3{}_{xx}=0 ,
$$
$$
 	\xi ^2{}_{uuu}+2 \xi ^3{}_{yuu}=0 ,
$$
$$
 	\xi ^2{}_{uuuu}+2 \xi ^3{}_{yuuu}=0 ,
$$
$$
	 \xi ^1{}_{uuuu}+2 \xi ^3{}_{xuuu}=0 ,
$$
$$
 	(g(u)+r(u)) \xi ^3{}_{u}=0 ,
$$
$$
 	\xi ^1{}_{yuu}+\xi ^2{}_{xuu}+\xi ^3{}_{xyu}=0 ,
$$
$$
 	 \xi ^1{}_{y}+\xi ^2{}_{x}+2 g(u) \xi ^3{}_{xy}=0 ,
$$
$$
 	\xi ^1{}_{y}+\xi ^2{}_{x}+2 r(u) \xi ^3{}_{xy}=0 ,
$$
$$
 	2 h' \xi ^3{}_{u}+(1+h(u)) \xi ^3{}_{uu}=0 ,
$$
$$
 	2 \left(1+r'\right) \xi ^3{}_{u}+r(u) \xi ^3{}_{uu}=0 ,
$$
$$
 	\xi ^1{}_{yuuu}+\xi ^2{}_{xuuu}+\xi ^3{}_{xyuu}=0 ,
$$
$$
 	8 h' \xi ^3{}_{u}+(1+4 h(u)) \xi ^3{}_{uu}=0 ,
$$
 $$
 	\left(11+16 r'\right) \xi ^3{}_{u}+8 r(u) \xi ^3{}_{uu}=0 ,
$$
$$
 	6 \xi ^3{}_{yyu}+2 \xi ^3{}_{xxu} -6 g(u) \xi ^3{}_{u}=0 ,
$$
$$
	6 \xi ^3{}_{yyu}+2 \xi ^3{}_{xxu} -2 r(u) \xi ^3{}_{u}=0 ,
$$
$$
	 2 \xi ^3{}_{yyu}+6 \xi ^3{}_{xxu}-2 g(u) \xi ^3{}_{u}=0 ,
$$
$$
	2 \xi ^3{}_{yyu}+6 \xi ^3{}_{xxu} -6 r(u) \xi ^3{}_{u}=0 ,
$$
$$
 	6 g(u) \xi ^3{}_{y}-4 \left(\xi ^3{}_{yyy}+\xi ^3{}_{xxy}\right)=0 ,
$$
$$
 	6 r(u) \xi ^3{}_{x}-4 \left(\xi ^3{}_{xyy}+\xi ^3{}_{xxx}\right)=0 ,
$$
$$
 	 6 g(u) \xi ^3{}_{u}-4 \left(3 \xi ^3{}_{yyu}+\xi ^3{}_{xxu}\right)=0 ,
$$
$$
 	6 r(u) \xi ^3{}_{u}-4 \left(\xi ^3{}_{yyu}+3 \xi ^3{}_{xxu}\right)=0 ,
$$
$$
 	 4 h(u) \xi ^3{}_{u}+2 g' \xi ^3{}_{u}+g(u) \xi ^3{}_{uu}=0 ,
$$
$$
 	\xi ^2{}_{u}+\left(1+2 g'\right) \xi ^3{}_{y}+2 g(u) \xi ^3{}_{yu}=0 ,
$$
$$
 	\xi ^2{}_{u}+\left(1+2 r'\right) \xi ^3{}_{y}+2 r(u) \xi ^3{}_{yu}=0 ,
$$
$$
  	\xi ^2{}_{uu}+2 h' \xi ^3{}_{y}+(1+2 h(u)) \xi ^3{}_{yu}=0 ,
$$
$$
 	\xi ^1{}_{uu}+2 h' \xi ^3{}_{x}+(1+2 h(u)) \xi ^3{}_{xu}=0 ,
$$
$$
 	11 h(u) \xi ^3{}_{u}+8 g' \xi ^3{}_{u}+4 g(u) \xi ^3{}_{uu}=0 ,
$$
$$
  	5 \xi ^1{}_{u}+\left(6+8 r'\right) \xi ^3{}_{x}+8 r(u) \xi ^3{}_{xu}=0 ,
$$
$$
  	 \left(3+8 h(u)+8 g'\right) \xi ^3{}_{u}+4 g(u) \xi ^3{}_{uu}=0 ,
$$
$$
 	\left(1+2 h(u)+4 r'\right) \xi ^3{}_{u}+2 r(u) \xi ^3{}_{uu}=0 ,
$$
$$
 	\left(2+3 h(u)+4 r'\right) \xi ^3{}_{u}+2 r(u) \xi ^3{}_{uu}=0 ,
$$
$$
 	\left(5+6 h(u)+4 r'\right) \xi ^3{}_{u}+2 r(u) \xi ^3{}_{uu}=0 ,
$$
$$
 	\left(1+2 h(u)+4 g'\right) \xi ^3{}_{u}+2 g(u) \xi ^3{}_{uu}=0 ,
$$
$$
 	3 \xi ^2{}_{u}+4 \left(\left(1+r'\right) \xi ^3{}_{y}+r(u) \xi ^3{}_{yu}\right)=0 ,
$$
$$
 	 \xi ^1{}_{yu}+\xi ^2{}_{xu}+2 g' \xi ^3{}_{xy}+2 g(u) \xi ^3{}_{xyu}=0 ,
$$
$$
 	\xi ^1{}_{yu}+\xi ^2{}_{xu}+2 r' \xi ^3{}_{xy}+2 r(u) \xi ^3{}_{xyu}=0 ,
$$
$$
  	\left(3+10 h(u)+4 g'\right) \xi ^3{}_{u}+2 g(u) \xi ^3{}_{uu}=0 ,
$$
$$
 	3 h^{\prime \prime } \xi ^3{}_{u}+3 h' \xi ^3{}_{uu}+(1+h(u)) \xi ^3{}_{uuu}=0 ,
$$
$$
 	3 r^{\prime \prime } \xi ^3{}_{u}+\left(2+3 r'\right) \xi ^3{}_{uu}+r(u) \xi ^3{}_{uuu}=0 ,
$$
$$
 	3 g^{\prime \prime } \xi ^3{}_{u}+\left(1+3 g'\right) \xi ^3{}_{uu}+g(u) \xi ^3{}_{uuu}=0 ,
$$
$$
  	  \xi ^1{}_{yuu}+\xi ^2{}_{xuu}+2 h' \xi ^3{}_{xy}+2 h(u) \xi ^3{}_{xyu}=0 ,
$$
$$
          \xi ^1{}_{u}+2 \left(h(u) \xi ^3{}_{x}+g' \xi ^3{}_{x}+g(u) \xi ^3{}_{xu}\right)=0 ,
$$
$$
 	\xi ^1{}_{u}+2 \left(h(u) \xi ^3{}_{x}+r' \xi ^3{}_{x}+r(u) \xi ^3{}_{xu}\right)=0 ,
$$
$$
  	3 \xi ^1{}_{u}+8 h(u) \xi ^3{}_{x}+4 g' \xi ^3{}_{x}+4 g(u) \xi ^3{}_{xu}=0 ,
$$
$$
   	2 g(u) \xi ^3{}_{y}-4 \left(\xi ^3{}_{yyy}+\xi ^3{}_{xxy}-r(u) \xi ^3{}_{y}\right)=0 ,
$$
$$
 	12 h^{\prime \prime } \xi ^3{}_{u}+12 h' \xi ^3{}_{uu}+(1+4 h(u)) \xi ^3{}_{uuu}=0 ,
$$
$$
  	\xi ^1{}_{uuu}+4 h^{\prime \prime } \xi ^3{}_{x}+8 h' \xi ^3{}_{xu}+4 h(u) \xi ^3{}_{xuu}=0 ,
$$
$$
   	4 g(u) \xi ^3{}_{x}-8 \left(\xi ^3{}_{xyy}+\xi ^3{}_{xxx}-r(u) \xi ^3{}_{x}\right)=0 ,
$$
$$
    	5 \xi ^2{}_{u}+12 h(u) \xi ^3{}_{y}+8 g' \xi ^3{}_{y}+8 g(u) \xi ^3{}_{yu}=0 ,
$$
$$
         	8 g(u) \xi ^3{}_{y}+4 r(u) \xi ^3{}_{y}-8 \left(\xi ^3{}_{yyy}+\xi ^3{}_{xxy}\right)=0 ,
$$
$$
	4 g(u) \xi ^3{}_{x}+2 r(u) \xi ^3{}_{x}-4 \left(\xi ^3{}_{xyy}+\xi ^3{}_{xxx}\right)=0 ,
$$
$$
	 \left(1+4 g'+4 r'\right) \xi ^3{}_{u}+2 (g(u)+r(u)) \xi ^3{}_{uu}=0 ,
$$
$$
    	2 g(u) \xi ^3{}_{u}-4 \left(3 \xi ^3{}_{yyu}+\xi ^3{}_{xxu}-r(u) \xi ^3{}_{u}\right)=0 ,
$$
$$
      	4 g(u) \xi ^3{}_{u}-8 \left(\xi ^3{}_{yyu}+3 \xi ^3{}_{xxu}-r(u) \xi ^3{}_{u}\right)=0 ,
$$
$$
      	12 r^{\prime \prime } \xi ^3{}_{u}+\left(13+12 r'\right) \xi ^3{}_{uu}+4 r(u) \xi ^3{}_{uuu}=0 ,
$$
$$
  	3 \xi ^2{}_{u}+\left(3+4 h(u)+2 g'\right) \xi ^3{}_{y}+2 g(u) \xi ^3{}_{yu}=0 ,
$$
$$
     	2 g(u) \xi ^3{}_{u}+6 r(u) \xi ^3{}_{u}-4 \left(3 \xi ^3{}_{yyu}+\xi ^3{}_{xxu}\right)=0 ,
$$
$$
 	8 g(u) \xi ^3{}_{u}+4 r(u) \xi ^3{}_{u}-8 \left(3 \xi ^3{}_{yyu}+\xi ^3{}_{xxu}\right)=0 ,
$$
$$
  	6 g(u) \xi ^3{}_{u}+2 r(u) \xi ^3{}_{u}-4 \left(\xi ^3{}_{yyu}+3 \xi ^3{}_{xxu}\right)=0 ,
$$
$$
 	4 g(u) \xi ^3{}_{u}+2 r(u) \xi ^3{}_{u}-4 \left(\xi ^3{}_{yyu}+3 \xi ^3{}_{xxu}\right)=0 ,
$$
$$
 	3 \xi ^1{}_{u}+2 \left(\left(1+3 h(u)+r'\right) \xi ^3{}_{x}+r(u) \xi ^3{}_{xu}\right)=0 ,
$$
$$
 	\left(5+12 g'+4 r'\right) \xi ^3{}_{u}+2 (3 g(u)+r(u)) \xi ^3{}_{uu}=0 ,
$$
$$
 	2 \left(1+3 g'+2 r'\right) \xi ^3{}_{u}+(3 g(u)+2 r(u)) \xi ^3{}_{uu}=0 ,
$$
$$
  	\xi ^2{}_{u}+2 \left(g' \xi ^3{}_{y}+r' \xi ^3{}_{y}+(g(u)+r(u)) \xi ^3{}_{yu}\right)=0 ,
$$
$$
 	\xi ^1{}_{u}+2 \left(g' \xi ^3{}_{x}+r' \xi ^3{}_{x}+(g(u)+r(u)) \xi ^3{}_{xu}\right)=0 ,
$$
$$
 	4 h^{\prime \prime \prime } \xi ^3{}_{u}+6 h^{\prime \prime } \xi ^3{}_{uu}+4 h' \xi ^3{}_{uuu}+h(u) \xi ^3{}_{uuuu}=0 ,
$$
$$
 	\xi ^2{}_{uu}+2 g^{\prime \prime } \xi ^3{}_{y}+\xi ^3{}_{yu}+4 g' \xi ^3{}_{yu}+2 g(u) \xi ^3{}_{yuu}=0 ,
$$
$$
	 \xi ^2{}_{uu}+2 r^{\prime \prime } \xi ^3{}_{y}+\xi ^3{}_{yu}+4 r' \xi ^3{}_{yu}+2 r(u) \xi ^3{}_{yuu}=0 ,
$$
$$
	 \xi ^1{}_{uu}+2 g^{\prime \prime } \xi ^3{}_{x}+\xi ^3{}_{xu}+4 g' \xi ^3{}_{xu}+2 g(u) \xi ^3{}_{xuu}=0 ,
$$
$$
  	3 \xi ^1{}_{yu}+3 \xi ^2{}_{xu}+2 \left(\left(1+r'\right) \xi ^3{}_{xy}+r(u) \xi ^3{}_{xyu}\right)=0 ,
$$
$$
  	 3 \xi ^2{}_{uu}+4 \left(r^{\prime \prime } \xi ^3{}_{y}+\left(1+2 r'\right) \xi ^3{}_{yu}+r(u) \xi ^3{}_{yuu}\right)=0 ,
$$
$$
    	 3 \xi ^2{}_{uu}+2 r^{\prime \prime } \xi ^3{}_{y}+5 \xi ^3{}_{yu}+4 r' \xi ^3{}_{yu}+2 r(u) \xi ^3{}_{yuu}=0 ,
$$
$$
 	5 \xi ^1{}_{uu}+4 r^{\prime \prime } \xi ^3{}_{x}+8 \xi ^3{}_{xu}+8 r' \xi ^3{}_{xu}+4 r(u) \xi ^3{}_{xuu}=0 ,
$$
$$
 	\xi ^2{}_{uuu}+2 h^{\prime \prime } \xi ^3{}_{y}+4 h' \xi ^3{}_{yu}+\xi ^3{}_{yuu}+2 h(u) \xi ^3{}_{yuu}=0 ,
$$
$$
 	\xi ^1{}_{uuu}+2 h^{\prime \prime } \xi ^3{}_{x}+4 h' \xi ^3{}_{xu}+\xi ^3{}_{xuu}+2 h(u) \xi ^3{}_{xuu}=0 ,
$$
$$
 	5 \xi ^1{}_{uu}+6 r^{\prime \prime } \xi ^3{}_{x}+7 \xi ^3{}_{xu}+12 r' \xi ^3{}_{xu}+6 r(u) \xi ^3{}_{xuu}=0 ,
$$
$$
 	h(u) \xi ^3{}_{u}+2 g' \xi ^3{}_{u}+2 r' \xi ^3{}_{u}+g(u) \xi ^3{}_{uu}+r(u) \xi ^3{}_{uu}=0 ,
$$
$$
  	3 \xi ^1{}_{u}+2 \left(\left(1+3 g'+r'\right) \xi ^3{}_{x}+(3 g(u)+r(u)) \xi ^3{}_{xu}\right)=0 ,
$$
$$
   	3 \xi ^1{}_{yu}+3 \xi ^2{}_{xu}+4 h(u) \xi ^3{}_{xy}+2 g' \xi ^3{}_{xy}+2 g(u) \xi ^3{}_{xyu}=0 ,
$$
$$
  	3 \left(3+8 r'\right) \xi ^3{}_{u}+14 r(u) \xi ^3{}_{uu}-4 \left(\xi ^3{}_{yyuu}+3 \xi ^3{}_{xxuu}\right)=0 ,
$$
$$
 	5 h(u) \xi ^3{}_{u}+2 g' \xi ^3{}_{u}+6 r' \xi ^3{}_{u}+g(u) \xi ^3{}_{uu}+3 r(u) \xi ^3{}_{uu}=0 ,
$$
$$
 	4 g^{\prime \prime \prime } \xi ^3{}_{u}+6 g^{\prime \prime } \xi ^3{}_{uu}+\xi ^3{}_{uuu}+4 g' \xi ^3{}_{uuu}+g(u) \xi ^3{}_{uuuu}=0 ,
$$
$$
  	 4 h' \xi ^3{}_{u}+3 r^{\prime \prime } \xi ^3{}_{u}+2 h(u) \xi ^3{}_{uu}+3 r' \xi ^3{}_{uu}+r(u) \xi ^3{}_{uuu}=0 ,
$$
$$
 	8 h' \xi ^3{}_{u}+3 g^{\prime \prime } \xi ^3{}_{u}+4 h(u) \xi ^3{}_{uu}+3 g' \xi ^3{}_{uu}+g(u) \xi ^3{}_{uuu}=0 ,
$$
$$
   	 4 r^{\prime \prime \prime } \xi ^3{}_{u}+6 r^{\prime \prime } \xi ^3{}_{uu}+7 \xi ^3{}_{uuu}+4 r' \xi ^3{}_{uuu}+r(u) \xi ^3{}_{uuuu}=0 ,
$$
$$
 	4 h^{\prime \prime \prime } \xi ^3{}_{u}+6 h^{\prime \prime } \xi ^3{}_{uu}+4 h' \xi ^3{}_{uuu}+\xi ^3{}_{uuuu}+h(u) \xi ^3{}_{uuuu}=0 ,
$$
$$
 	4 h(u) \xi ^3{}_{u}+4 g' \xi ^3{}_{u}+6 r' \xi ^3{}_{u}+2 g(u) \xi ^3{}_{uu}+3 r(u) \xi ^3{}_{uu}=0 ,
$$
$$ 
 	\xi ^1{}_{yuuu}+\xi ^2{}_{xuuu}+2 h^{\prime \prime } \xi ^3{}_{xy}+4 h' \xi ^3{}_{xyu}+2 h(u) \xi ^3{}_{xyuu}=0 ,
$$
$$
  	9 h(u) \xi ^3{}_{u}+12 g' \xi ^3{}_{u}+7 g(u) \xi ^3{}_{uu}-6 \xi ^3{}_{yyuu}-2 \xi ^3{}_{xxuu}=0 ,
$$
$$
  	\xi ^2{}_{uuuu}+4 \left(h^{\prime \prime \prime } \xi ^3{}_{y}+3 h^{\prime \prime } \xi ^3{}_{yu}+3 h' \xi ^3{}_{yuu}+h(u) \xi ^3{}_{yuuu}\right)=0 ,
$$
$$
 	\xi ^1{}_{uuuu}+4 \left(h^{\prime \prime \prime } \xi ^3{}_{x}+3 h^{\prime \prime } \xi ^3{}_{xu}+3 h' \xi ^3{}_{xuu}+h(u) \xi ^3{}_{xuuu}\right)=0 ,
$$
$$
 	26 h' \xi ^3{}_{u}+6 g^{\prime \prime } \xi ^3{}_{u}+13 h(u) \xi ^3{}_{uu}+6 g' \xi ^3{}_{uu}+2 g(u) \xi ^3{}_{uuu}=0 ,
$$
$$
  	4 \xi ^1{}_{u}+\left(5+12 r'\right) \xi ^3{}_{x}+14 r(u) \xi ^3{}_{xu}-4 \xi ^3{}_{xyyu}-4 \xi ^3{}_{xxxu}=0 ,
$$
$$
 	16 h^{\prime \prime \prime } \xi ^3{}_{u}+24 h^{\prime \prime } \xi ^3{}_{uu}+16 h' \xi ^3{}_{uuu}+\xi ^3{}_{uuuu}+4 h(u) \xi ^3{}_{uuuu}=0 ,
$$
$$
 	 \left(1+4 g'\right) \xi ^3{}_{x}+4 g(u) \xi ^3{}_{xu}+2 r(u) \xi ^3{}_{xu}-4 \xi ^3{}_{xyyu}-4 \xi ^3{}_{xxxu}=0 ,
$$
$$
    	2 \xi ^2{}_{u}+5 h(u) \xi ^3{}_{y}+6 g' \xi ^3{}_{y}+7 g(u) \xi ^3{}_{yu}-2 \xi ^3{}_{yyyu}-2 \xi ^3{}_{xxyu}=0 ,
$$
$$
      	\xi ^2{}_{uu}+2 \left(h' \xi ^3{}_{y}+r^{\prime \prime } \xi ^3{}_{y}+h(u) \xi ^3{}_{yu}+2 r' \xi ^3{}_{yu}+r(u) \xi ^3{}_{yuu}\right)=0 ,
$$
$$
         	\xi ^1{}_{uu}+2 \left(h' \xi ^3{}_{x}+g^{\prime \prime } \xi ^3{}_{x}+h(u) \xi ^3{}_{xu}+2 g' \xi ^3{}_{xu}+g(u) \xi ^3{}_{xuu}\right)=0 ,
$$
$$
 	\xi ^1{}_{uu}+2 \left(h' \xi ^3{}_{x}+r^{\prime \prime } \xi ^3{}_{x}+h(u) \xi ^3{}_{xu}+2 r' \xi ^3{}_{xu}+r(u) \xi ^3{}_{xuu}\right)=0 ,
$$
$$
	3 \xi ^2{}_{u}+4 h(u) \xi ^3{}_{y}+2 g' \xi ^3{}_{y}+6 r' \xi ^3{}_{y}+2 g(u) \xi ^3{}_{yu}+6 r(u) \xi ^3{}_{yu}=0 ,
$$
$$
 	4 h' \xi ^3{}_{u}+6 g^{\prime \prime } \xi ^3{}_{u}+\xi ^3{}_{uu}+2 h(u) \xi ^3{}_{uu}+6 g' \xi ^3{}_{uu}+2 g(u) \xi ^3{}_{uuu}=0 ,
$$
$$
 	4 h' \xi ^3{}_{u}+6 r^{\prime \prime } \xi ^3{}_{u}+\xi ^3{}_{uu}+2 h(u) \xi ^3{}_{uu}+6 r' \xi ^3{}_{uu}+2 r(u) \xi ^3{}_{uuu}=0 ,
$$
$$
       	\xi ^1{}_{uuu}+2 g^{\prime \prime \prime } \xi ^3{}_{x}+6 g^{\prime \prime } \xi ^3{}_{xu}+\xi ^3{}_{xuu}+6 g' \xi ^3{}_{xuu}+2 g(u) \xi ^3{}_{xuuu}=0 ,
$$
$$
  	\left(1+8 g'\right) \xi ^3{}_{u}+4 g(u) \xi ^3{}_{uu}+2 r(u) \xi ^3{}_{uu}-4 \xi ^3{}_{yyuu}-12 \xi ^3{}_{xxuu}=0 ,
$$
$$
 	3 \xi ^1{}_{yuu}+3 \xi ^2{}_{xuu}+2 \left(r^{\prime \prime } \xi ^3{}_{xy}+\left(1+2 r'\right) \xi ^3{}_{xyu}+r(u) \xi ^3{}_{xyuu}\right)=0 ,
$$
$$
 	6 h' \xi ^3{}_{u}+6 r^{\prime \prime } \xi ^3{}_{u}+2 \xi ^3{}_{uu}+3 h(u) \xi ^3{}_{uu}+6 r' \xi ^3{}_{uu}+2 r(u) \xi ^3{}_{uuu}=0 ,
$$
$$ 
  	h(u) \xi ^3{}_{y}+2 r' \xi ^3{}_{y}+g(u) \xi ^3{}_{yu}+2 r(u) \xi ^3{}_{yu}-2 \xi ^3{}_{yyyu}-2 \xi ^3{}_{xxyu}=0 ,
$$
$$
  	h(u) \xi ^3{}_{u}+4 r' \xi ^3{}_{u}+g(u) \xi ^3{}_{uu}+2 r(u) \xi ^3{}_{uu}-6 \xi ^3{}_{yyuu}-2 \xi ^3{}_{xxuu}=0 ,
$$
$$
  	3 \xi ^2{}_{uuu}+2 r^{\prime \prime \prime } \xi ^3{}_{y}+6 r^{\prime \prime } \xi ^3{}_{yu}+5 \xi ^3{}_{yuu}+6 r' \xi ^3{}_{yuu}+2 r(u) \xi ^3{}_{yuuu}=0 ,
$$
$$
 	5 \xi ^1{}_{uuu}+2 r^{\prime \prime \prime } \xi ^3{}_{x}+6 r^{\prime \prime } \xi ^3{}_{xu}+9 \xi ^3{}_{xuu}+6 r' \xi ^3{}_{xuu}+2 r(u) \xi ^3{}_{xuuu}=0 ,
$$
$$
  	12 h' \xi ^3{}_{u}+6 r^{\prime \prime } \xi ^3{}_{u}+5 \xi ^3{}_{uu}+6 h(u) \xi ^3{}_{uu}+6 r' \xi ^3{}_{uu}+2 r(u) \xi ^3{}_{uuu}=0 ,
$$
$$
 	5 \xi ^2{}_{uu}+4 \left(4 h' \xi ^3{}_{y}+g^{\prime \prime } \xi ^3{}_{y}+4 h(u) \xi ^3{}_{yu}+2 g' \xi ^3{}_{yu}+g(u) \xi ^3{}_{yuu}\right)=0 ,
$$
$$
  	\xi ^2{}_{uuuu}+2 h^{\prime \prime \prime } \xi ^3{}_{y}+6 h^{\prime \prime } \xi ^3{}_{yu}+6 h' \xi ^3{}_{yuu}+\xi ^3{}_{yuuu}+2 h(u) \xi ^3{}_{yuuu}=0 ,
$$
$$
 	\xi ^1{}_{uuuu}+2 h^{\prime \prime \prime } \xi ^3{}_{x}+6 h^{\prime \prime } \xi ^3{}_{xu}+6 h' \xi ^3{}_{xuu}+\xi ^3{}_{xuuu}+2 h(u) \xi ^3{}_{xuuu}=0 ,
$$
$$
    	3 \xi ^1{}_{uu}+4 \left(2 h' \xi ^3{}_{x}+g^{\prime \prime } \xi ^3{}_{x}+2 h(u) \xi ^3{}_{xu}+2 g' \xi ^3{}_{xu}+g(u) \xi ^3{}_{xuu}\right)=0 ,
$$
$$
 	3 \xi ^1{}_{uu}+2 \left(5 h' \xi ^3{}_{x}+g^{\prime \prime } \xi ^3{}_{x}+5 h(u) \xi ^3{}_{xu}+2 g' \xi ^3{}_{xu}+g(u) \xi ^3{}_{xuu}\right)=0 ,
$$
$$
   	20 h' \xi ^3{}_{u}+6 g^{\prime \prime } \xi ^3{}_{u}+3 \xi ^3{}_{uu}+10 h(u) \xi ^3{}_{uu}+6 g' \xi ^3{}_{uu}+2 g(u) \xi ^3{}_{uuu}=0 ,
$$
$$
     	5 \xi ^2{}_{uu}+2 \left(7 h' \xi ^3{}_{y}+3 g^{\prime \prime } \xi ^3{}_{y}+7 h(u) \xi ^3{}_{yu}+6 g' \xi ^3{}_{yu}+3 g(u) \xi ^3{}_{yuu}\right)=0 ,
$$
$$
 	16 h' \xi ^3{}_{u}+12 g^{\prime \prime } \xi ^3{}_{u}+3 \xi ^3{}_{uu}+8 h(u) \xi ^3{}_{uu}+12 g' \xi ^3{}_{uu}+4 g(u) \xi ^3{}_{uuu}=0 ,
$$
$$
 	\left(5+12 g'+4 r'\right) \xi ^3{}_{u}+6 g(u) \xi ^3{}_{uu}-4 \left(\xi ^3{}_{yyuu}+3 \xi ^3{}_{xxuu}-r(u) \xi ^3{}_{uu}\right)=0 ,
$$
$$
  	3 \xi ^1{}_{uu}+2 \left(3 h' \xi ^3{}_{x}+r^{\prime \prime } \xi ^3{}_{x}+\xi ^3{}_{xu}+3 h(u) \xi ^3{}_{xu}+2 r' \xi ^3{}_{xu}+r(u) \xi ^3{}_{xuu}\right)=0 ,
$$
$$
 	3 \xi ^2{}_{uu}+4 h' \xi ^3{}_{y}+2 g^{\prime \prime } \xi ^3{}_{y}+3 \xi ^3{}_{yu}+4 h(u) \xi ^3{}_{yu}+4 g' \xi ^3{}_{yu}+2 g(u) \xi ^3{}_{yuu}=0 ,
$$
$$
 	2 \xi ^1{}_{u}+\left(3+6 g'+2 r'\right) \xi ^3{}_{x}+6 g(u) \xi ^3{}_{xu}+4 r(u) \xi ^3{}_{xu}-4 \xi ^3{}_{xyyu}-4 \xi ^3{}_{xxxu}=0 ,
$$
$$
 	5 h(u) \xi ^3{}_{u}+2 g' \xi ^3{}_{u}+6 r' \xi ^3{}_{u}+2 g(u) \xi ^3{}_{uu}+3 r(u) \xi ^3{}_{uu}-6 \xi ^3{}_{yyuu}-2 \xi ^3{}_{xxuu}=0 ,
$$
$$
 	18 g^{\prime \prime } \xi ^3{}_{u}+6 r^{\prime \prime } \xi ^3{}_{u}+5 \xi ^3{}_{uu}+18 g' \xi ^3{}_{uu}+6 r' \xi ^3{}_{uu}+6 g(u) \xi ^3{}_{uuu}+2 r(u) \xi ^3{}_{uuu}=0 ,
$$
$$
 	3 \xi ^1{}_{yuu}+3 \xi ^2{}_{xuu}+4 h' \xi ^3{}_{xy}+2 g^{\prime \prime } \xi ^3{}_{xy}+4 h(u) \xi ^3{}_{xyu}+4 g' \xi ^3{}_{xyu}+2 g(u) \xi ^3{}_{xyuu}=0 ,
$$
$$
 	6 h^{\prime \prime } \xi ^3{}_{u}+4 r^{\prime \prime \prime } \xi ^3{}_{u}+6 h' \xi ^3{}_{uu}+6 r^{\prime \prime } \xi ^3{}_{uu}+2 h(u) \xi ^3{}_{uuu}+4 r' \xi ^3{}_{uuu}+r(u) \xi ^3{}_{uuuu}=0 ,
$$
$$
 	\xi ^2{}_{u}+3 h(u) \xi ^3{}_{y}+g' \xi ^3{}_{y}+3 r' \xi ^3{}_{y}+2 g(u) \xi ^3{}_{yu}+3 r(u) \xi ^3{}_{yu}-2 \xi ^3{}_{yyyu}-2 \xi ^3{}_{xxyu}=0 ,
$$
$$
 	g(u)^2 \xi ^3{}_{y}-2 f' \xi ^3{}_{y}+2 \eta _{yu}-2 f(u) \xi ^3{}_{yu}-3 \xi ^2{}_{yy}-\xi ^2{}_{xx}-2 g(u) \left(\xi ^3{}_{yyy}-\xi ^3{}_{xxy}\right)=0 ,
$$
$$
 	\xi ^1{}_{yy}-r(u)^2 \xi ^3{}_{x}+2 f' \xi ^3{}_{x}-2 \eta _{xu}+2 f(u) \xi ^3{}_{xu}+3 \xi ^1{}_{xx}+2 r(u) \left(\xi ^3{}_{xyy}+\xi ^3{}_{xxx}\right)=0 ,
$$
\begin{multline*}
  	2 \left(3 g^{\prime \prime } \xi ^3{}_{x}+r^{\prime \prime } \xi ^3{}_{x}+\xi ^3{}_{xu}+6 g' \xi ^3{}_{xu}+2 r' \xi ^3{}_{xu}+3 g(u) \xi ^3{}_{xuu}+r(u) \xi ^3{}_{xuu}\right)\\
	+3 \xi ^1{}_{uu}=0 ,
\end{multline*}
\begin{multline*}
  	g(u) \xi ^1{}_{y}-2 \xi ^1{}_{yyy}+r(u) \xi ^2{}_{x}-4 f' \xi ^3{}_{xy}+4 \eta _{xyu}-4 f(u) \xi ^3{}_{xyu}-2 \xi ^2{}_{xyy}-2 \xi ^1{}_{xxy}\\
	-2 \xi ^2{}_{xxx}=0 ,
\end{multline*}
\begin{multline*}
 	f(u) \xi ^3{}_{u}+\xi ^3{}_{t}-g(u) \xi ^3{}_{yy}-2 r(u) \xi ^3{}_{yy}+\xi ^3{}_{yyyy}-4 \xi ^1{}_{x}-7 r(u) \xi ^3{}_{xx}+2 \xi ^3{}_{xxyy}\\
	+\xi ^3{}_{xxxx}=0 ,
\end{multline*}
\begin{multline*}
 	f(u) \xi ^3{}_{u}+\xi ^3{}_{t}-4 \xi ^2{}_{y}-7 g(u) \xi ^3{}_{yy}+\xi ^3{}_{yyyy}-2 g(u) \xi ^3{}_{xx}-r(u) \xi ^3{}_{xx}+2 \xi ^3{}_{xxyy}\\
	+\xi ^3{}_{xxxx}=0 ,
\end{multline*}
\begin{multline*}
 	9 h^{\prime \prime } \xi ^3{}_{y}+g^{\prime \prime \prime } \xi ^3{}_{y}+18 h' \xi ^3{}_{yu}+3 g^{\prime \prime } \xi ^3{}_{yu}+9 h(u) \xi ^3{}_{yuu}+3 g' \xi ^3{}_{yuu}+g(u) \xi ^3{}_{yuuu}\\
	+\frac{5}{2} \xi ^2{}_{uuu}=0 ,
\end{multline*}
\begin{multline*}
 	5 h^{\prime \prime } \xi ^3{}_{x}+g^{\prime \prime \prime } \xi ^3{}_{x}+10 h' \xi ^3{}_{xu}+3 g^{\prime \prime } \xi ^3{}_{xu}+5 h(u) \xi ^3{}_{xuu}+3 g' \xi ^3{}_{xuu}+g(u) \xi ^3{}_{xuuu}\\
	+\frac{3}{2} \xi ^1{}_{uuu}=0 ,
\end{multline*}
\begin{multline*}
 	42 h^{\prime \prime } \xi ^3{}_{u}+4 g^{\prime \prime \prime } \xi ^3{}_{u}+42 h' \xi ^3{}_{uu}+6 g^{\prime \prime } \xi ^3{}_{uu}+14 h(u) \xi ^3{}_{uuu}+4 g' \xi ^3{}_{uuu}\\
	+g(u) \xi ^3{}_{uuuu}=0 ,
\end{multline*}
\begin{multline*}
 	4 f' \xi ^3{}_{u}-2 \eta _{uu}+2 f(u) \xi ^3{}_{uu}+4 \xi ^2{}_{yu}+3 \xi ^3{}_{yy}+2 h(u) \xi ^3{}_{yy}+4 \xi ^1{}_{xu}+\xi ^3{}_{xx}\\
	+6 h(u) \xi ^3{}_{xx}=0 ,
\end{multline*}
\begin{multline*}
 	10 h' \xi ^3{}_{u}+3 g^{\prime \prime } \xi ^3{}_{u}+9 r^{\prime \prime } \xi ^3{}_{u}+5 h(u) \xi ^3{}_{uu}+3 g' \xi ^3{}_{uu}+9 r' \xi ^3{}_{uu}+g(u) \xi ^3{}_{uuu}\\
	+3 r(u) \xi ^3{}_{uuu}=0 ,
\end{multline*}
\begin{multline*}
 	g(u) r(u) \xi ^3{}_{y}-2 f' \xi ^3{}_{y}+2 \eta _{yu}-2 f(u) \xi ^3{}_{yu}-\xi ^2{}_{yy}-2 r(u) \xi ^3{}_{yyy}-4 \xi ^1{}_{xy}-3 \xi ^2{}_{xx}\\
	-2 r(u) \xi ^3{}_{xxy}=0 ,
\end{multline*}
\begin{multline*}
   	 \frac{3}{2} \xi ^1{}_{uuu}+ 3 h^{\prime \prime } \xi ^3{}_{x}+r^{\prime \prime \prime } \xi ^3{}_{x}+6 h' \xi ^3{}_{xu}+3 r^{\prime \prime } \xi ^3{}_{xu}+\xi ^3{}_{xuu}+3 h(u) \xi ^3{}_{xuu}+3 r' \xi ^3{}_{xuu}\\
	 +r(u) \xi ^3{}_{xuuu}=0 ,
\end{multline*}
\begin{multline*}
	30 h^{\prime \prime } \xi ^3{}_{u}+8 g^{\prime \prime \prime } \xi ^3{}_{u}+30 h' \xi ^3{}_{uu}+12 g^{\prime \prime } \xi ^3{}_{uu}+3 \xi ^3{}_{uuu}+10 h(u) \xi ^3{}_{uuu}+8 g' \xi ^3{}_{uuu}\\
	+2 g(u) \xi ^3{}_{uuuu}=0 ,
\end{multline*}
\begin{multline*}
   	3 \xi ^2{}_{uuu}+4 h^{\prime \prime } \xi ^3{}_{y}+2 g^{\prime \prime \prime } \xi ^3{}_{y}+8 h' \xi ^3{}_{yu}+6 g^{\prime \prime } \xi ^3{}_{yu}+3 \xi ^3{}_{yuu}+4 h(u) \xi ^3{}_{yuu}+6 g' \xi ^3{}_{yuu}\\
	+2 g(u) \xi ^3{}_{yuuu}=0 ,
\end{multline*}
\begin{multline*}
 	18 h^{\prime \prime } \xi ^3{}_{u}+8 r^{\prime \prime \prime } \xi ^3{}_{u}+18 h' \xi ^3{}_{uu}+12 r^{\prime \prime } \xi ^3{}_{uu}+5 \xi ^3{}_{uuu}+6 h(u) \xi ^3{}_{uuu}+8 r' \xi ^3{}_{uuu}\\
	+2 r(u) \xi ^3{}_{uuuu}=0 ,
\end{multline*}
\begin{multline*}
 3 \xi ^1{}_{yy}+2 f' \xi ^3{}_{x}+4 \xi ^2{}_{xy}+\xi ^1{}_{xx}+g(u) \left(-r(u) \xi ^3{}_{x}+2 \left(\xi ^3{}_{xyy}+\xi ^3{}_{xxx}\right)\right)\\
 -2 \eta _{xu}+2 f(u) \xi ^3{}_{xu}=0 ,
\end{multline*}
\begin{multline*}
 	2 \left(h^{\prime \prime } \xi ^3{}_{y}+r^{\prime \prime \prime } \xi ^3{}_{y}+2 h' \xi ^3{}_{yu}+3 r^{\prime \prime } \xi ^3{}_{yu}+h(u) \xi ^3{}_{yuu}+3 r' \xi ^3{}_{yuu}\right.\\
	\left.+r(u) \xi ^3{}_{yuuu}\right)+\xi ^2{}_{uuu}=0 ,
\end{multline*}
\begin{multline*}  
  	 2 h' \xi ^3{}_{y}+g^{\prime \prime } \xi ^3{}_{y}+3 r^{\prime \prime } \xi ^3{}_{y}+2 h(u) \xi ^3{}_{yu}+2 g' \xi ^3{}_{yu}+6 r' \xi ^3{}_{yu}+g(u) \xi ^3{}_{yuu}\\
	+3 r(u) \xi ^3{}_{yuu}+\frac{3}{2} \xi ^2{}_{uu}=0 ,
\end{multline*}
\begin{multline*}
 	2 f^{\prime \prime } \xi ^3{}_{y}+4 f' \xi ^3{}_{yu}-2 \eta _{yuu}+2 f(u) \xi ^3{}_{yuu}+\xi ^2{}_{yyu}+\xi ^3{}_{yyy}+4 \xi ^1{}_{xyu}+3 \xi ^2{}_{xxu}+\xi ^3{}_{xxy}\\
	-\frac{1}{2}r(u) \xi ^2{}_{u}-\frac{1}{2}g(u) \xi ^3{}_{y}=0 ,
\end{multline*}
\begin{multline*}
 	4 f' \xi ^3{}_{u}-r(u)^2 \xi ^3{}_{u}-2 \eta _{uu}+2 f(u) \xi ^3{}_{uu}+\xi ^3{}_{yy}+2 r' \xi ^3{}_{yy}+8 \xi ^1{}_{xu}+3 \xi ^3{}_{xx}+6 r' \xi ^3{}_{xx}\\
	+2 r(u) \left(\xi ^3{}_{yyu}+3 \xi ^3{}_{xxu}\right)=0 ,
 \end{multline*}
\begin{multline*}
 	3 \eta _{uu}-6 f' \xi ^3{}_{u}-3 f(u) \xi ^3{}_{uu}-2 \xi ^3{}_{yy}-2 r' \xi ^3{}_{yy}-2 r(u) \xi ^3{}_{yyu}-12 \xi ^1{}_{xu}-6 \xi ^3{}_{xx}\\
	-6 r' \xi ^3{}_{xx}-6 r(u) \xi ^3{}_{xxu}=0 ,
 \end{multline*}
\begin{multline*}
 	f(u) \xi ^3{}_{u}+\xi ^3{}_{t}-2 \xi ^2{}_{y}-2 g(u) \xi ^3{}_{yy}-3 r(u) \xi ^3{}_{yy}+\xi ^3{}_{yyyy}-2 \xi ^1{}_{x}-3 g(u) \xi ^3{}_{xx}\\
	-2 r(u) \xi ^3{}_{xx}+2 \xi ^3{}_{xxyy}+\xi ^3{}_{xxxx}=0 ,
 \end{multline*}
\begin{multline*}
 	6 f^{\prime \prime } \xi ^3{}_{uu}+4 f' \xi ^3{}_{uuu}-\eta _{uuuu}+f(u) \xi ^3{}_{uuuu}+\xi ^3{}_{yyuu}+4 \xi ^1{}_{xuuu}+3 \xi ^3{}_{xxuu}\\
	-\frac{1}{4}\left(1-16 f^{\prime \prime \prime }\right) \xi ^3{}_{u}-\frac{1}{2}r(u) \xi ^3{}_{uu}=0 ,
 \end{multline*}
\begin{multline*}
 	3 \eta _{uu}-6 f' \xi ^3{}_{u}-3 f(u) \xi ^3{}_{uu}-12 \xi ^2{}_{yu}-12 h(u) \xi ^3{}_{yy}-6 g' \xi ^3{}_{yy}-6 g(u) \xi ^3{}_{yyu}\\
	-4 h(u) \xi ^3{}_{xx}-2 g' \xi ^3{}_{xx}-2 g(u) \xi ^3{}_{xxu}=0 ,
 \end{multline*}
 \begin{multline*}
 	4 f' \xi ^3{}_{u}-2 \eta _{uu}+2 f(u) \xi ^3{}_{uu}+4 \xi ^2{}_{yu}+3 \xi ^3{}_{yy}+2 g' \xi ^3{}_{yy}+4 \xi ^1{}_{xu}+\xi ^3{}_{xx}+6 g' \xi ^3{}_{xx}\\
	+g(u) \left(-r(u) \xi ^3{}_{u}+2 \xi ^3{}_{yyu}+6 \xi ^3{}_{xxu}\right)=0 ,
 \end{multline*}
 \begin{multline*}
 	2 f' \xi ^3{}_{u}-\eta _{uu}+f(u) \xi ^3{}_{uu}+4 \xi ^2{}_{yu}+3 h(u) \xi ^3{}_{yy}+3 g' \xi ^3{}_{yy}+h(u) \xi ^3{}_{xx}+g' \xi ^3{}_{xx}\\
	+ g(u) \left(3 \xi ^3{}_{yyu}+\xi ^3{}_{xxu}\right)-\frac{1}{2}g(u)^2 \xi ^3{}_{u}=0 ,
 \end{multline*}
\begin{multline*}
 	6 \xi ^1{}_{yyu}-g(u) \xi ^1{}_{u}-2 h(u) r(u) \xi ^3{}_{x}+4 f^{\prime \prime } \xi ^3{}_{x}+8 f' \xi ^3{}_{xu}-4 \eta _{xuu}+4 f(u) \xi ^3{}_{xuu}\\
	+8 \xi ^2{}_{xyu}+4 h(u) \xi ^3{}_{xyy}+2 \xi ^1{}_{xxu}+4 h(u) \xi ^3{}_{xxx}=0 ,
\end{multline*}
\begin{multline*}
	2 f' \xi ^3{}_{u}-\eta _{uu}+f(u) \xi ^3{}_{uu}+2 \xi ^2{}_{yu}+h(u) \xi ^3{}_{yy}+3 r' \xi ^3{}_{yy}+3 r(u) \xi ^3{}_{yyu}+2 \xi ^1{}_{xu}\\
	+3 h(u) \xi ^3{}_{xx}+r' \xi ^3{}_{xx}+r(u) \xi ^3{}_{xxu} -\frac{1}{2}g(u) r(u) \xi ^3{}_{u}=0 ,
\end{multline*}
\begin{multline*}
 	2 f' \xi ^3{}_{u}-\eta _{uu}+f(u) \xi ^3{}_{uu}+2 \xi ^2{}_{yu}+g' \xi ^3{}_{yy}+3 r' \xi ^3{}_{yy}+g(u) \xi ^3{}_{yyu}+3 r(u) \xi ^3{}_{yyu}\\
	+2 \xi ^1{}_{xu}+3 g' \xi ^3{}_{xx}+r' \xi ^3{}_{xx}+3 g(u) \xi ^3{}_{xxu}+r(u) \xi ^3{}_{xxu}=0 ,
 \end{multline*}
\begin{multline*}
 	2 h(u) \xi ^1{}_{y}+2 g(u) \xi ^1{}_{yu}-4 \xi ^1{}_{yyyu}+\xi ^2{}_{x}+2 r(u) \xi ^2{}_{xu}-8 f^{\prime \prime } \xi ^3{}_{xy}-16 f' \xi ^3{}_{xyu}\\
	+8 \eta _{xyuu}-8 f(u) \xi ^3{}_{xyuu}-4 \xi ^2{}_{xyyu}-4 \xi ^1{}_{xxyu}-4 \xi ^2{}_{xxxu}=0 ,
 \end{multline*}
\begin{multline*}
 	\xi ^1{}_{u}-4 \xi ^1{}_{yyuu}+\xi ^3{}_{x}-8 f^{\prime \prime \prime } \xi ^3{}_{x}-24 f^{\prime \prime } \xi ^3{}_{xu}+2 r(u) \left(\xi ^1{}_{uu}+\xi ^3{}_{xu}\right)-24 f' \xi ^3{}_{xuu}\\
	+8 \eta _{xuuu}-8 f(u) \xi ^3{}_{xuuu}-4 \xi ^3{}_{xyyu}-12 \xi ^1{}_{xxuu}-4 \xi ^3{}_{xxxu}=0 ,
\end{multline*}
\begin{multline*}
 	\eta _{t}-\eta (x,y,t,u) f'-f(u)^2 \xi ^3{}_{u}-g(u) \eta _{yy}+\eta _{yyyy}-r(u) \eta _{xx}+2 \eta _{xxyy}+\eta _{xxxx}\\
	+f(u) \left(\eta _{u}-\xi ^3{}_{t}+g(u) \xi ^3{}_{yy}-\xi ^3{}_{yyyy}+r(u) \xi ^3{}_{xx}-2 \xi ^3{}_{xxyy}-\xi ^3{}_{xxxx}\right)=0 ,
 \end{multline*}
\begin{multline*}
 	4 f^{\prime \prime \prime } \xi ^3{}_{y}-r(u) \xi ^2{}_{uu}-h(u) \xi ^3{}_{y}-g(u) \xi ^3{}_{yu}+12 f^{\prime \prime } \xi ^3{}_{yu}+12 f' \xi ^3{}_{yuu}-4 \eta _{yuuu}\\
	+4 f(u) \xi ^3{}_{yuuu}+2 \xi ^2{}_{yyuu}+2 \xi ^3{}_{yyyu}+8 \xi ^1{}_{xyuu}+6 \xi ^2{}_{xxuu}+2 \xi ^3{}_{xxyu}-\frac{1}{2}\xi ^2{}_{u}=0 ,
 \end{multline*}
\begin{multline*}
 	6 f^{\prime \prime } \xi ^3{}_{u}+6 f' \xi ^3{}_{uu}-2 \eta _{uuu}+2 f(u) \xi ^3{}_{uuu}+4 \xi ^2{}_{yuu}+2 h' \xi ^3{}_{yy}+3 \xi ^3{}_{yyu}+4 \xi ^1{}_{xuu}\\
	+6 h' \xi ^3{}_{xx}+\xi ^3{}_{xxu}+h(u) \left(-r(u) \xi ^3{}_{u}+2 \xi ^3{}_{yyu}+6 \xi ^3{}_{xxu}\right)-\frac{1}{2}g(u) \xi ^3{}_{u}=0 ,
 \end{multline*}
\begin{multline*}
 	2 f^{\prime \prime } \xi ^3{}_{x}+4 f' \xi ^3{}_{xu}-2 \eta _{xuu}+2 f(u) \xi ^3{}_{xuu}+4 \xi ^2{}_{xyu}+2 g' \xi ^3{}_{xyy}+\xi ^1{}_{xxu}+2 g' \xi ^3{}_{xxx}\\
	-r(u) g' \xi ^3{}_{x}+3 \xi ^1{}_{yyu}-\frac{1}{2}g(u) \left(\xi ^3{}_{x}+2 r(u) \xi ^3{}_{xu}-4 \left(\xi ^3{}_{xyyu}+\xi ^3{}_{xxxu}\right)\right)=0 ,
 \end{multline*}
 \begin{multline*}
 	 3 \xi ^1{}_{yyu}+6 f^{\prime \prime } \xi ^3{}_{x}+12 f' \xi ^3{}_{xu}-6 \eta _{xuu}+6 f(u) \xi ^3{}_{xuu}+2 \xi ^3{}_{xyy}+2 r' \xi ^3{}_{xyy}+9 \xi ^1{}_{xxu}\\
	+2 \xi ^3{}_{xxx}+ r(u)^2 \xi ^3{}_{xu}-\frac{1}{2}r(u) \left(2 \xi ^1{}_{u}+\left(3+2 r'\right) \xi ^3{}_{x}-4 \left(\xi ^3{}_{xyyu}+\xi ^3{}_{xxxu}\right)\right)\\
	+2 r' \xi ^3{}_{xxx}=0 ,
\end{multline*}
\begin{multline*}
 	2 f^{\prime \prime } \xi ^3{}_{y}+4 f' \xi ^3{}_{yu}-g(u) \left(r' \xi ^3{}_{y}+r(u) \xi ^3{}_{yu}\right)+2 f(u) \xi ^3{}_{yuu}-h(u) r(u) \xi ^3{}_{y}\\
	-2 \eta _{yuu}+\xi ^2{}_{yyu}+2 r' \xi ^3{}_{yyy}+2 r(u) \xi ^3{}_{yyyu}+4 \xi ^1{}_{xyu}+3 \xi ^2{}_{xxu}+2 r' \xi ^3{}_{xxy}\\
	+2 r(u) \xi ^3{}_{xxyu}=0 ,
	\end{multline*}
\begin{multline*}
 	h(u)^2 \xi ^3{}_{u}+2 g(u) h' \xi ^3{}_{u}-4 f^{\prime \prime \prime } \xi ^3{}_{u}-6 f^{\prime \prime } \xi ^3{}_{uu}-4 f' \xi ^3{}_{uuu}+\eta _{uuuu}-f(u) \xi ^3{}_{uuuu}\\
	-12 h' \xi ^3{}_{yyu}-2 h^{\prime \prime } \xi ^3{}_{xx}-4 h' \xi ^3{}_{xxu}+h(u) \left(g(u) \xi ^3{}_{uu}-2 \left(3 \xi ^3{}_{yyuu}+\xi ^3{}_{xxuu}\right)\right)\\
	-6 h^{\prime \prime } \xi ^3{}_{yy}-4 \xi ^2{}_{yuuu}=0 ,
 \end{multline*}
 \begin{multline*}
 	-\xi ^1{}_{t}+g(u) \xi ^1{}_{yy}-\xi ^1{}_{yyyy}-\eta _{x}+2 r(u) f' \xi ^3{}_{x}-2 r(u) \eta _{xu}-4 f' \xi ^3{}_{xyy}+4 \eta _{xyyu}\\
	+f(u) \left(-\xi ^1{}_{u}+\xi ^3{}_{x}+2 r(u) \xi ^3{}_{xu}-4 \xi ^3{}_{xyyu}-4 \xi ^3{}_{xxxu}\right)-\xi ^1{}_{xxxx}+r(u) \xi ^1{}_{xx}\\
	-2 \xi ^1{}_{xxyy}-4 f' \xi ^3{}_{xxx}+4 \eta _{xxxu}=0 ,
 \end{multline*}
\begin{multline*}
 	18 f^{\prime \prime } \xi ^3{}_{u}+18 f' \xi ^3{}_{uu}-6 \eta _{uuu}+6 f(u) \xi ^3{}_{uuu}+2 r^{\prime \prime } \xi ^3{}_{yy}+5 \xi ^3{}_{yyu}+4 r' \xi ^3{}_{yyu}\\
	+12 r' \xi ^3{}_{xxu}- r(u)^2 \xi ^3{}_{uu}-\frac{1}{2} r(u) \left(\left(5+4 r'\right) \xi ^3{}_{u}-4 \left(\xi ^3{}_{yyuu}+3 \xi ^3{}_{xxuu}\right)\right)\\
	+15 \xi ^3{}_{xxu}+24 \xi ^1{}_{xuu}+6 r^{\prime \prime } \xi ^3{}_{xx}=0 ,
 \end{multline*}
 \begin{multline*}
 	3 \xi ^1{}_{yyuu}-r(u) h' \xi ^3{}_{x}+2 f^{\prime \prime \prime } \xi ^3{}_{x}+6 f^{\prime \prime } \xi ^3{}_{xu}+6 f' \xi ^3{}_{xuu}-2 \eta _{xuuu}+2 f(u) \xi ^3{}_{xuuu}\\
		+4 \xi ^2{}_{xyuu}-\frac{1}{2}g(u) \xi ^1{}_{uu}-\frac{1}{2}h(u) \left(\xi ^1{}_{u}+\xi ^3{}_{x}+2 r(u) \xi ^3{}_{xu}-4 \xi ^3{}_{xyyu}-4 \xi ^3{}_{xxxu}\right)\\
		+2 h' \xi ^3{}_{xyy}+\xi ^1{}_{xxuu}+2 h' \xi ^3{}_{xxx}=0 ,
 \end{multline*}
\begin{multline*}
 	6 g^{\prime \prime } \xi ^3{}_{xx}-2 r(u) g' \xi ^3{}_{u}+6 f^{\prime \prime } \xi ^3{}_{u}+6 f' \xi ^3{}_{uu}-2 \eta _{uuu}+2 f(u) \xi ^3{}_{uuu}+4 \xi ^2{}_{yuu}\\
	+2 g^{\prime \prime } \xi ^3{}_{yy}+\xi ^3{}_{xxu}+12 g' \xi ^3{}_{xxu}-\frac{1}{2}g(u) \left(\xi ^3{}_{u}+2 r(u) \xi ^3{}_{uu}-4 \xi ^3{}_{yyuu}-12 \xi ^3{}_{xxuu}\right)\\
	+3 \xi ^3{}_{yyu}+4 g' \xi ^3{}_{yyu}+4 \xi ^1{}_{xuu}=0 ,
 \end{multline*}
\begin{multline*}
 	2 g(u) f' \xi ^3{}_{y}-\xi ^2{}_{t}-2 h(u) \eta _{y}-2 g(u) \eta _{yu}+g(u) \xi ^2{}_{yy}-4 f' \xi ^3{}_{yyy}+4 \eta _{yyyu}-\xi ^2{}_{yyyy}\\
	-f(u) \left(\xi ^2{}_{u}-2 h(u) \xi ^3{}_{y}-2 g(u) \xi ^3{}_{yu}+4 \xi ^3{}_{yyyu}+4 \xi ^3{}_{xxyu}\right)-2 \xi ^2{}_{xxyy}-\xi ^2{}_{xxxx}\\
	+r(u) \xi ^2{}_{xx}-4 f' \xi ^3{}_{xxy}+4 \eta _{xxyu}=0 ,
 \end{multline*}
\begin{multline*}
 	6 f^{\prime \prime } \xi ^3{}_{y}-g(u)^2 \xi ^3{}_{yu}+12 f' \xi ^3{}_{yu}-6 \eta _{yuu}+6 f(u) \xi ^3{}_{yuu}+9 \xi ^2{}_{yyu}+4 h(u) \xi ^3{}_{yyy}\\
	+2 g' \xi ^3{}_{xxy}-g(u) \left(\xi ^2{}_{u}+3 h(u) \xi ^3{}_{y}+g' \xi ^3{}_{y}-2 \xi ^3{}_{yyyu}-2 \xi ^3{}_{xxyu}\right)\\
	+2 g' \xi ^3{}_{yyy}+3 \xi ^2{}_{xxu}+4 h(u) \xi ^3{}_{xxy}=0 ,
 \end{multline*}
\begin{multline*}
 	2 f^{\prime \prime \prime } \xi ^3{}_{y}+6 f^{\prime \prime } \xi ^3{}_{yu}+6 f' \xi ^3{}_{yuu}-2 \eta _{yuuu}+2 f(u) \xi ^3{}_{yuuu}+3 \xi ^2{}_{yyuu}+2 h' \xi ^3{}_{yyy}\\
	+2 h' \xi ^3{}_{xxy}- h(u)^2 \xi ^3{}_{y}-\frac{1}{2}h(u) \left(\xi ^2{}_{u}+2 g(u) \xi ^3{}_{yu}-4 \left(\xi ^3{}_{yyyu}+\xi ^3{}_{xxyu}\right)\right)\\
	+\xi ^2{}_{xxuu}-\frac{1}{2}g(u) \left(\xi ^2{}_{uu}+2 h' \xi ^3{}_{y}\right)=0 ,
\end{multline*}
\begin{multline*}
	 2 r(u) \xi ^1{}_{x}-\eta (x,y,t,u) r'-r(u) \xi ^3{}_{t}+g(u) r(u) \xi ^3{}_{yy}-2 f' \xi ^3{}_{yy}+2 \eta _{yyu}\\
	-r(u) \xi ^3{}_{yyyy} -4 \xi ^1{}_{xyy}+r(u)^2 \xi ^3{}_{xx}-6 f' \xi ^3{}_{xx}+6 \eta _{xxu}-2 r(u) \xi ^3{}_{xxyy}-4 \xi ^1{}_{xxx}\\
	-r(u) \xi ^3{}_{xxxx}-f(u) \left(r(u) \xi ^3{}_{u}+2 \xi ^3{}_{yyu}+6 \xi ^3{}_{xxu}\right)=0 ,
\end{multline*}
\begin{multline*}
	g(u)^2 \xi ^3{}_{yy} -\eta (x,y,t,u) g'-g(u) \xi ^3{}_{t}+2 g(u) \xi ^2{}_{y}-6 f' \xi ^3{}_{yy}+6 \eta _{yyu}-4 \xi ^2{}_{yyy}\\
	 +g(u) r(u) \xi ^3{}_{xx}-2 f' \xi ^3{}_{xx}+2 \eta _{xxu}-f(u) \left(g(u) \xi ^3{}_{u}+6 \xi ^3{}_{yyu}+2 \xi ^3{}_{xxu}\right)\\
	 -g(u) \xi ^3{}_{yyyy}-4 \xi ^2{}_{xxy}-2 g(u) \xi ^3{}_{xxyy}-g(u) \xi ^3{}_{xxxx}=0 ,
\end{multline*}
\begin{multline*}
 	9 f^{\prime \prime } \xi ^3{}_{u}+9 f' \xi ^3{}_{uu}+3 f(u) \xi ^3{}_{uuu}+12 \xi ^2{}_{yuu}+15 h' \xi ^3{}_{yy}+3 g^{\prime \prime } \xi ^3{}_{yy}+6 g' \xi ^3{}_{yyu}\\
	-3 \eta _{uuu}+15 h(u) \xi ^3{}_{yyu}+5 h' \xi ^3{}_{xx}+g^{\prime \prime } \xi ^3{}_{xx}+5 h(u) \xi ^3{}_{xxu}+2 g' \xi ^3{}_{xxu}-\frac{1}{2}g(u)^2 \xi ^3{}_{uu}\\
	-\frac{1}{2}g(u) \left(5 h(u) \xi ^3{}_{u}+2 g' \xi ^3{}_{u}-6 \xi ^3{}_{yyuu}-2 \xi ^3{}_{xxuu}\right)=0 ,
 \end{multline*}
\begin{multline*}
 	4 r(u) h' \xi ^3{}_{u}-16 f^{\prime \prime \prime } \xi ^3{}_{u}+g(u) \xi ^3{}_{uu}-24 f^{\prime \prime } \xi ^3{}_{uu}-16 f' \xi ^3{}_{uuu}+4 \eta _{uuuu}-8 \xi ^2{}_{yuuu}\\
	-4 f(u) \xi ^3{}_{uuuu}-4 h^{\prime \prime } \xi ^3{}_{yy}-8 h' \xi ^3{}_{yyu}-6 \xi ^3{}_{yyuu}-8 \xi ^1{}_{xuuu}-12 h^{\prime \prime } \xi ^3{}_{xx}-2 \xi ^3{}_{xxuu}\\
	-24 h' \xi ^3{}_{xxu}+2 h(u) \left(\xi ^3{}_{u}+r(u) \xi ^3{}_{uu}-2 \xi ^3{}_{yyuu}-6 \xi ^3{}_{xxuu}\right)=0 ,
 \end{multline*}
\begin{multline*}
	3 f^{\prime \prime } \xi ^3{}_{u}+3 f' \xi ^3{}_{uu}-\eta _{uuu}+f(u) \xi ^3{}_{uuu}+2 \xi ^2{}_{yuu}+h' \xi ^3{}_{yy}+3 r^{\prime \prime } \xi ^3{}_{yy}+6 r' \xi ^3{}_{yyu}\\
	+3 h' \xi ^3{}_{xx}+r^{\prime \prime } \xi ^3{}_{xx}+2 r' \xi ^3{}_{xxu}+r(u) \xi ^3{}_{xxuu}-\frac{1}{2}g(u) \left(2 r' \xi ^3{}_{u}+r(u) \xi ^3{}_{uu}\right)\\
	-\frac{1}{2}h(u) \left(r(u) \xi ^3{}_{u}-2 \left(\xi ^3{}_{yyu}+3 \xi ^3{}_{xxu}\right)\right)+3 r(u) \xi ^3{}_{yyuu}+2 \xi ^1{}_{xuu}=0 ,
 \end{multline*}
\begin{multline*}
 	2 \xi ^1{}_{x}-\eta _{u}-\xi ^3{}_{t}+g(u) \xi ^3{}_{yy}-4 f^{\prime \prime } \xi ^3{}_{yy}-8 f' \xi ^3{}_{yyu}+4 \eta _{yyuu}-4 f(u) \xi ^3{}_{yyuu}-\xi ^3{}_{yyyy}\\
	-8 \xi ^1{}_{xyyu}-12 f^{\prime \prime } \xi ^3{}_{xx}+r(u) \left(4 f' \xi ^3{}_{u}-2 \eta _{uu}+2 f(u) \xi ^3{}_{uu}+4 \xi ^1{}_{xu}+\xi ^3{}_{xx}\right)\\
	-24 f' \xi ^3{}_{xxu}+12 \eta _{xxuu}-12 f(u) \xi ^3{}_{xxuu}-2 \xi ^3{}_{xxyy}-8 \xi ^1{}_{xxxu}-\xi ^3{}_{xxxx}=0 ,
 \end{multline*}
\begin{multline*}
 	2 g(u) \xi ^2{}_{yu}-\eta (x,y,t,u) h'+2 g(u) f' \xi ^3{}_{u}-g(u) \eta _{uu}+f(u) g(u) \xi ^3{}_{uu}-6 f^{\prime \prime } \xi ^3{}_{yy}\\
	-12 f' \xi ^3{}_{yyu}+6 \eta _{yyuu}-6 f(u) \xi ^3{}_{yyuu}-4 \xi ^2{}_{yyyu}-2 f^{\prime \prime } \xi ^3{}_{xx}-4 f' \xi ^3{}_{xxu}+2 \eta _{xxuu}\\
	-h(u) \left(\eta _{u}+\xi ^3{}_{t}-2 \xi ^2{}_{y}-g(u) \xi ^3{}_{yy}+\xi ^3{}_{yyyy}-r(u) \xi ^3{}_{xx}+2 \xi ^3{}_{xxyy}+\xi ^3{}_{xxxx}\right)\\
	-2 f(u) \xi ^3{}_{xxuu}-4 \xi ^2{}_{xxyu}=0 ,
\end{multline*}

\bibliographystyle{plain}
\bibliography{Bibliography}

\begin{thebibliography}{10}

\bibitem{BluAncChe2k9}
G.~W. Bluman, A.~F. Cheviakov, and S.~C. Anco.
\newblock {\em Applications of symmetry methods to partial differential
  equations}.
\newblock Applied Mathematical Sciences. Springer, New York, 2009.

\bibitem{BluKu89}
G.~W. Bluman and S.~Kumei.
\newblock {\em Symmetries and differential equations}.
\newblock Springer, New York, 1989.

\bibitem{BoChoHwa99}
B.~M. Boghosian, C.~C. Chow, and T.~Hwa.
\newblock Hydrodynamics of the {K}uramoto-{S}ivashinsky equation in two
  dimensions.
\newblock {\em Phys. Rev. Lett.}, 83(25):5262--5265, 1999.

\bibitem{BruGaIbra2k9}
M.~S. Bruz\'on, M.~L. Gandarias, and N.~H. Ibragimov.
\newblock Self-adjoint sub-classes of generalized thin film equations.
\newblock {\em J. Math. Anal. Appl.}, 357:307--313, 2009.

\bibitem{CoKroTaRo76}
B.~Cohen, J.~Krommes, W.~Tang, and M.~Rosenbluth.
\newblock Nonlinear saturation of the dissipative trapped-ion mode by mode
  coupling.
\newblock {\em Nucl. Fus.}, 16:971--992, 1976.

\bibitem{CoMu89}
R.~Conte and M.~Musette.
\newblock Painlev\'e analysis and {B}\:acklund tranformation in the
  {K}uramoto-{S}ivashinsky equation.
\newblock {\em J. Phys. A: Math. Gen.}, 22:169--177, 1989.

\bibitem{CroHo93}
M.~C. Cross and P.~C. Hohenberg.
\newblock Pattern formation outside of equilibrium.
\newblock {\em Rev. Mod. Phys.}, 65(3):851--1123, July 1993.

\bibitem{De2k9}
A.~Demirkaya.
\newblock The existence of a global attractor for a {K}uramoto-{S}ivashinsky
  type equation in 2d.
\newblock {\em Discrete Contin. Dyn. Syst}, pages 198--207, 2009.

\bibitem{Dimas2k8}
S.~Dimas.
\newblock Partial differential equations, algebraic computing and nonlinear
  systems.
\newblock Ph.{D}. {T}hesis, University of Patras, Patras, Greece, October 2008.

\bibitem{DiTs2k5a}
S.~Dimas and D.~Tsoubelis.
\newblock {SYM}: A new symmetry-finding package for {M}athematica.
\newblock In N.H. Ibragimov, C.~Sophocleous, and P.A. Damianou, editors, {\em
  The 10$^{th}$ International Conference in {MO}dern {GR}oup {AN}alysis}, pages
  64--70, Nicosia, 2005. University of Cyprus.

\bibitem{DiTs2k6}
S.~Dimas and D.~Tsoubelis.
\newblock A new {M}athematica-based program for solving overdetermined systems
  of {PDEs}.
\newblock In Y.~Papegay, editor, {\em Applied {Mathematica}, {E}lectronic
  {P}roceedings of the {E}ighth {I}nternational {M}athematica {S}ymposium
  (IMS'06)}, Avignon, France, 2006. France: INRIA.
\newblock ISBN 2-7261-1289-7.

\bibitem{DroZhaLuWa99}
Jason~T. Drotar, Y.-P. Zhao, T.-M. Lu, and G.-C. Wang.
\newblock Numerical analysis of the noisy {K}uramoto-{S}ivashinsky equation in
  $2+1$ dimensions.
\newblock {\em Phys. Rev. E}, 59(1):177--185, 1999.

\bibitem{GaMiPo2k9}
V.~A. Galaktionov, E.~Mitidieri, and S.~I. Pohozaev.
\newblock On global solutions and blow-up for {K}uramoto-{S}ivashinsky-type
  models, and well-posed {B}urnett equations.
\newblock {\em Nonlinear Anal.}, 70:2930--2952, 2009.

\bibitem{HaZha95}
T.~Halpin-Healy and Y.-C. Zhang.
\newblock Kinetic roughening phenomena, stochastic growth, directed polymers
  and all that. {A}spects of multidisciplinary statistical mechanics.
\newblock {\em Phys. Rep.}, 254:215--414, 1995.

\bibitem{Hydon2k}
P.~E. Hydon.
\newblock {\em Symmetry Methods for Differential Equations}.
\newblock Cambridge Texts in Applied Mathematics. Cambridge University Press,
  Cambridge, 1$^{st}$ edition, 2000.

\bibitem{HyNi86}
J.~M. Hyman and B.~Nicolaenko.
\newblock The {K}uramoto-{S}ivashinsky equation: A bridge between pdes and
  dynamical systems.
\newblock {\em Physica D}, 18:113--126, 1986.

\bibitem{HyNiZa86}
J.~M. Hyman, B.~Nicolaenko, and S.~Zaleski.
\newblock Order and complexity in the {K}uramoto-{S}ivashinsky model of weakly
  turbulent interfaces.
\newblock {\em Physica D}, 23:265--292, 1986.

\bibitem{Ibra2k1}
N.~H. Ibragimov.
\newblock {\em Transformation Groups Applied to Mathematical Physics}.
\newblock Mathematics and its Applications. Springer, 1$^{st}$ edition,
  November 2001.

\bibitem{Ibra2k7}
N.H. Ibragimov.
\newblock A new conservation theorem.
\newblock {\em J. Math. Anal. App}, 333:311--328, 2007.

\bibitem{Ibr2k11a}
N.H. Ibragimov.
\newblock Nonlinear self-adjointness and conservation laws.
\newblock {\em J. Phys. A: Math. Theor.}, 44:432002, 2011.

\bibitem{Ibr2k11b}
N.H. Ibragimov.
\newblock Nonlinear self-adjointness in constructing conservation laws.
\newblock {\em Archives of ALGA}, 7/8:1--90, 2011.

\bibitem{JaHaPa93}
C.~Jayaprakash, F.~Hayot, and R.~Pandit.
\newblock Universal properties of the two-dimensional {K}uramoto-{S}ivashinsky
  equation.
\newblock {\em Phys. Rev. Lett.}, 71(1):12--15, 1993.

\bibitem{KeNiSco90}
I.~G. Kevrekidis, B.~Nicolaenko, and J.~C. Scovel.
\newblock Back in the saddle again: A computer assisted study of the
  {K}uramoto-{S}ivashinsky equation.
\newblock {\em SIAM J. Appl. Math.}, 50:760--790, 1990.

\bibitem{Ku78}
Y.~Kuramoto.
\newblock Diffusion-induced chaos in reactions systems.
\newblock {\em Prog. Theor. Phys.}, 64:346--367, 1978.

\bibitem{KuTsu75}
Y.~Kuramoto and T.~Tsuzuki.
\newblock On the formation of dissipative structures in reaction--diffusion
  systems.
\newblock {\em Prog. Theor. Phys.}, 54:687--699, 1975.

\bibitem{KuTsu76}
Y.~Kuramoto and T.~Tsuzuki.
\newblock Persistent propagation of concentration waves in dissipative media
  far from thermal equilibrium.
\newblock {\em Prog. Theor. Phys.}, 55(2):356--369, 1976.

\bibitem{LaMaRuTa75}
R.~La{Q}uey, S.~Mahajan, P.~Rutherford, and W.~Tang.
\newblock Nonlinear saturation of the trapped-ion mode.
\newblock {\em Phys. Rev. Lett.}, 34:391--394, 1975.

\bibitem{LuBo98}
J.~Lundbek~Hansen and T.~Bohr.
\newblock Fractal tracer distributions in turbulent field theories.
\newblock {\em Physica D}, 118:40--48, July 1998.

\bibitem{LvoLePaPro93}
V.~S. L'vov, V.~V. Lebedev, M.~Paton, and I.~Procaccia.
\newblock Proof of scale invariant solutions in the {K}ardar-{P}arisi-{Z}hang
  and {K}uramoto-{S}ivashinsky equations in $1+1$ dimensions: analytical and
  numerical results.
\newblock {\em Nonlinearity}, 6:25--47, 1993.

\bibitem{MaBa97}
M.~A. Makeev and A.~L. Barab\'asi.
\newblock Ion-induced effective surface diffusion in ion sputtering.
\newblock {\em Appl. Phys. Lett.}, 71:2800--2802, 1997.

\bibitem{Ma88}
P.~Manneville.
\newblock The {K}uramoto-{S}ivashinsky equation: a progress report.
\newblock In J.~Wesfreid, H.~R. Brand, P.~Manneville, G.~Albinet, and
  N.~Boccara, editors, {\em Propagation in Systems Far from Equilibrium:
  Proceedings of the Workshop}, Springer Series in Synergetics, pages 265--280,
  Les Houches, France, March 1988. Springer.

\bibitem{Mi86}
D.~Michelson.
\newblock Steady solutions of the {K}uramoto-{S}ivashinsky equation.
\newblock {\em Physica D}, 19:89--111, 1986.

\bibitem{MiSi77}
D.~M. Michelson and G.~I. Sivashinsky.
\newblock Nonlinear analysis of hydrodynamic instability in laminar flames --
  {II}. {N}umerical experiments.
\newblock {\em Acta Astronaut.}, 4:1207--1221, 1977.

\bibitem{NaAha2k11}
Mehdi Nadjafikhah and Fatemeh Ahangari.
\newblock Symmetry reduction of two-dimensional damped
  {K}uramoto--{S}ivashinsky equation.
\newblock {\em Commun. Theor. Phys.}, 56:211--217, 2011.

\bibitem{NaAha2k12}
Mehdi Nadjafikhah and Fatemeh Ahangari.
\newblock Lie symmetry analysis of the two-dimensional generalized
  {K}uramoto-{S}ivashinsky equation.
\newblock {\em Math. Sci.}, 6(1):3, 2012.

\bibitem{NiScheTe85}
B.~Nicolaenko, B.~Scheurer, and R.~Temam.
\newblock Attractors for the {K}uramoto-{S}ivashinsky equations.
\newblock {\em Physica D}, 16:155--183, 1985.

\bibitem{Olver2k}
P.~J. Olver.
\newblock {\em Applications of Lie Groups to Differential Equations}, volume
  107 of {\em Graduate Texts in Mathematics}.
\newblock Springer, New York, 2$^{nd}$ edition, 2000.

\bibitem{Ovsiannikov82}
L.~Ovsiannikov.
\newblock {\em Group Analysis of Differential Equations}.
\newblock Academic Press, 1$^{st}$ edition, June 1982.
\newblock 432 pages.

\bibitem{RoKru95}
M.~Rost and J.~Krug.
\newblock Anisotropic {K}uramoto-{S}ivashinsky equation for surface growth and
  erosion.
\newblock {\em Phys. Rev. Lett.}, 75:3894--3897, 1995.

\bibitem{Si77}
G.~I. Sivashinsky.
\newblock Nonlinear analysis of hydrodynamic instability in laminar flames --
  {I}. {D}erivation of basic equations.
\newblock {\em Acta Astronaut.}, 4:1177--1206, June 1977.

\bibitem{Si80}
G.~I. Sivashinsky.
\newblock On flame propagation under conditions of stoichiometry.
\newblock {\em SIAM J. Appl. Math.}, 39:67--82, 1980.

\bibitem{SiMi80}
G.~I. Sivashinsky and D.~Michelson.
\newblock On irregular wavy flow of a liquid film down a vertical plane.
\newblock {\em Prog. Theor. Phys.}, 63:2112--2114, 1980.

\bibitem{Stephani90}
H.~Stephani.
\newblock {\em Differential Equations: Their Solution Using Symmetries}.
\newblock Cambridge University Press, Cambridge, 1$^{st}$ edition, 1990.
\newblock Editor: MacCallum, Malcolm.

\bibitem{Te97}
R.~Temam.
\newblock {\em Infinite-dimensional dynamical systems in mechanics and
  physics}, volume~68 of {\em Applied Mathematical Sciences}.
\newblock Springer-Verlag, New York, second edition, 1997.

\bibitem{ThuFri86}
O.~Thuai and U.~Frisch.
\newblock {\em Natural boundary in the {K}uramoto-{S}ivashinsky model}, pages
  327--336.
\newblock Les Ulis : Ed. de Physique, Les Houches, March 1986.

\bibitem{EoM22687}
R.~W. Wittenberg.
\newblock {\em Encyclopaedia of {M}athematics, {S}upplement {III}}, pages
  230--233.
\newblock Kluwer, 2002.

\bibitem{WiHo99}
R.~W. Wittenberg and P.~Holmes.
\newblock Scale and space localization in the {K}uramoto--{S}ivashinsky
  equation.
\newblock {\em Chaos}, 9(2):452--465, 1999.

\end{thebibliography}

\end{document}